\documentclass[final]{gSSR2e}

\RequirePackage[english]{babel}

\def\prooffont{\normalfont}

\DeclareSymbolFont{msbm}{U}{msb}{m}{n}
\DeclareMathSymbol{\N}{\mathalpha}{msbm}{'116}
\DeclareMathSymbol{\R}{\mathalpha}{msbm}{'122}
\DeclareMathSymbol{\Q}{\mathalpha}{msbm}{'121}

\def\Om{\Omega}
\def\Omclo{\overline{\Omega}}

\def\pdi{\partial_i}
\def\pdj{\partial_j}
\def\pdk{\partial_k}

\def\limn{\underset{n \to \infty}{\text{lim}}}

\def\sup{\text{sup}}
\def\supp{\text{supp}}
\def\dist{\text{dist}}

\newcommand{\sharpbracket}[1]{\langle #1 \rangle}

\newcommand{\Mu}[1][u]{M^{[#1]}}

\newcommand{\Muts}[2]{{M_{#2}^{[#1]}}^2}
\newcommand{\Nu}[1][u]{N^{[#1]}}

\def\mN{\mathbf N}

\newcommand{\subE}{\widetilde{E}}
\newcommand{\mMRes}[1][\subE]{\mathbf M^{#1}}
\newcommand{\mOmRes}[1][\subE]{\widetilde{\mathbf \Om}}

\newcommand{\mFRes}[1][\subE]{\widetilde{\mathcal F}}
\newcommand{\mFtRes}[2][\subE]{\widetilde{\mathcal F}_{#2}}
\newcommand{\mXtRes}[2][\subE]{\mathbf X_{#2}^{#1}}
\newcommand{\mPxRes}[2][\subE]{\mathbb P_{#2}^{#1}}
\def\mFResInftCompl{\widetilde{\mathcal F}_\infty}

\def\mOm{\mathbf \Om}

\def\pairpot{\Psi}

\newcommand{\Aij}[2]{A^{(#1,#2)}}

\newcommand{\xib}[1]{x^{(#1)}}
\def\funcA{A}

\def\GammaN{\Gamma^{(N)}}

\def\cutoff{\phi}
\def\unitnormal{n}
\def\ellip{\gamma}
\def\lifetime{\mathcal X}
\def\Id{\mathbf{1}}
\def\capE{\text{cap}_\mathcal E}

\def\shiftop{\theta}

\def\Ltil{\hat{L}}

\def\minw{\wedge}

\def\maxw{\vee}

\def\DNeu{\mathcal D_{\text{Neu}}}

\def\ttau{\widetilde{\tau}}

\def\Ito{It\^{o}}

\def\conditionname{Condition}
\newtheorem{condition}{\conditionname}[section]

\newcommand*{\affaddr}[1]{#1} 
\newcommand*{\affmark}[1][*]{\textsuperscript{#1}}

\title{Skorokhod decomposition for a reflected $\mathcal L^p$-strong Feller diffusion with singular drift}

\begin{document}
\author{ Benedict Baur \affmark[1,a], Martin Grothaus \affmark[1,b] \\
\affaddr{\affmark[1] Department of Mathematics, University of Kaiserslautern, P.O. Box 3049, 67653 Kaiserslautern}\\
\thanks{\affmark[a] benedictbaur@gmail.com (Corresponding author)}
\thanks{\affmark[b] grothaus@mathematik.uni-kl.de}
}
\received{received October 2016}

%
%
%
%

\maketitle

\begin{abstract}

We construct Skorokhod decompositions for diffusions with singular drift and reflecting boundary behavior on open subsets of $\R^d$ with $C^2$-smooth boundary except for a sufficiently small set. 
This decomposition holds almost surely under the path measures of the process for every starting point from an explicitly known set. This set is characterized by the boundary smoothness and the singularities of the drift term.

We apply modern methods of Dirichlet form theory and $\mathcal L^p$-strong Feller processes. These tools have been approved as useful for the pointwise analysis of stochastic processes with singular drift and various boundary conditions. Furthermore, we apply Sobolev space theorems and elliptic regularity results to prove regularity properties of potentials related to surface measures. These are important ingredients for the pointwise construction of the boundary local time of the diffusions under consideration. 

As an application we construct stochastic dynamics for particle systems with hydrodynamic and pair interaction. Our approach allows highly singular potentials like Lennard-Jones potentials and position-dependent diffusion coefficients and thus the treatment of physically reasonable models.
\end{abstract}

\classcode{Primary 60J55 ,  31C25; Secondary 60J60 , 82C22}
\keywords{Skorokhod decomposition, Reflecting boundary behavior, Dirichlet forms, Diffusion processes, Local time and additive functionals, Interacting particle systems}

%


%
\section{Introduction}

In this article we construct Skorokhod decompositions for diffusion processes with variable diffusion coefficients, singular - possibly discontinuous -  drifts and reflecting boundary behavior. The construction is based on the $\mathcal L^p$-strong Feller diffusions we constructed in \cite{BG13}. These diffusions are associated with gradient Dirichlet forms and were constructed using methods of \cite{BGS13} and elliptic regularity results.
We emphasize that we construct the Skorokhod decomposition for every starting point in an explicitly specified set $E_1$ rather than just for quasi-every starting point. This overcomes a common drawback of Dirichlet form methods, namely that one has at first only statements for quasi-every starting point. The set $E_1$ is naturally related with the drift coefficients and boundary smoothness, see Theorem \ref{TheoDiffProcess} below. In applications this set can often be explicitly identified. 
We furthermore identify the constructed processes as weak solutions to singular stochastic differential equations (SDE).
All results hold under the same assumptions as in \cite{BG13}. These assumptions allow highly singular drift.

Let us now describe how we analyze the boundary behavior of the processes: The constructed processes in \cite{BG13} solve the martingale problem for a class of functions containing $C^2$-smooth function with homogenous Neumann-type boundary condition, see Theorem \ref{TheoDiffProcess}, below. This gives only limited information on the boundary behavior.
To see the reflecting boundary behavior, we need to enlarge the class of functions for which we get the martingale solution property. For this we construct the local time at a later to be specified boundary part. 

The local time is constructed as an additive functional of the process $(\mathbf X_t)_{t \ge 0}$. For the pointwise analysis it is essential to construct the local time as a functional that behaves well for every starting point in $E_1$.  We apply a construction result for additive functionals of Fukushima, Oshima and Takeda (\cite{FOT11}), see Theorem \ref{TheoremConstrTildeS} below. 

The construction of the local time is based on boundedness properties of $\alpha$-potentials of surface measures at compact boundary parts. We prove these properties by identifying the potentials as weak solutions to elliptic partial differential equations and by using our elliptic regularity result obtained in \cite{BG13}, see Theorem \ref{TheoFiniteEnergyIntegral} below. This is a new approach for applying \cite[Theo.~5.1.6]{FOT11}. In previous works these boundedness properties were checked using estimates on the resolvent kernels, see \cite{FT95}, \cite{FT96} and the explainations below. In the case of purely local assumptions on the coefficients it is more natural to take the approach via elliptic regularity results and Sobolev space theorems. As a byproduct we get even H\"older continuity and Sobolev space regularity of the potentials.

Using the local time we show a semimartingale decomposition of $(u(\mathbf X_t))_{t \ge 0}$ for $u \in C^2_c(E_1)$. This decomposition is valid under the path measure $\mathbb P_x$ for all $x \in E_1$, the set of admissible starting points, see Theorem \ref{TheoMartingaleWithLocalTime} and \eqref{EqSkoro} below. By localization techniques we obtain a Skorokhod decomposition (or semimartingale decomposition) for the process itself and can identify it as a weak solution to an SDE with reflection at the boundary, see Theorem \ref{TheoLocalMartingalePx} and Theorem \ref{TheoExWeakSolution} below.

In the pointwise setting additional care has to be taken due to possible singularities of the drifts. To handle these singularities, the $\mathcal L^p$-strong Feller property of the resolvent is important, see e.g. Theorem \ref{TheoremLpIntegraldt} below. Furthermore, we would like to note that it is a common principle to transform "`almost-everywhere"' or "`quasi-everywhere"' statements to everywhere (on $E_1$) statements by making use of the absolute continuity of the transition semigroup of kernels of a given stochastic process. However, this gives in general statements for strict positive time $t > 0$ only. To get a complete statement for $t \in [0,\infty)$ one has to do additional effort, e.g. the pointwise construction of the local time.

In Section \ref{SecIPS} we apply these results to interesting models of Mathematical Physics. Since our approach allows variable diffusion coefficients and strongly singular drifts we can handle physically reasonable settings. We construct stochastic dynamics for particle systems with hydrodynamic and direct pair interaction. Our results allow the treatment of potentials with strong repulsive singularities like potentials of Lennard-Jones type.
Particle systems with hydrodynamic interaction are a typical example for physical models with multiplicative noise, i.e., the diffusion matrix depends on the position of the particles.

Let us now introduce precisely our setting and assumptions. We then state the main results of this article. First we recall the results of \cite{BG13} on which our work is based.

We fix a $d \in \N$ with $d \ge 2$. Let $\Om \subset \R^d$ open and $A : \Omclo \to \R^{d \times d}$ a measurable mapping of symmetric elliptic matrices.
Let $\varrho : \Omclo \to \R^+_0$ be measurable with $\varrho > 0$ $dx$-a.e.

We consider the pre-Dirichlet form
\begin{equation}
\mathcal E(u,v) = \int_\Om (A \nabla u, \nabla v) \, d \mu, 
u,v \in \mathcal D := \{ u \in C_c(\Omclo) \, | \, u \in H^{1,1}_{\text{loc}}(\Om), \mathcal E(u,u) < \infty \}, \label{EqGradientDirForm}
\end{equation}
in the Hilbert space $L^2(\Omclo,\mu)$, where $\mu := \varrho dx$, $dx$ the Lebesgue measure on $\R^d$. As usual $L^p(\Omclo,\mu)$, $1 \le p < \infty$ ($p=\infty$), denotes the space of equivalence classes of $p$-integrable (essentially bounded) functions. By $C(\Omclo)$ we denote the space of continuous functions on $\Omclo$, the subindex $c$ marks that the functions have compact support in $\Omclo$. $H^{m,p}(\Om)$, $m \in \N_0$, $1 \le p \le \infty$ denotes the Sobolev space of $m$-times weakly differentiable functions with $L^p(\Om,dx)$ regularity. The subindex loc marks that the integrability is assumed to be local. As scalar product $(\cdot, \cdot)$ we take the euclidean scalar product on $\R^d$.

\begin{condition} \label{Cond1Matrix}
For each $x \in \Omclo$ the matrix $A(x)$ is symmetric and strictly elliptic, i.e., there exists an $\ellip(x) > 0$ such that
\begin{align*}
\gamma(x) (\xi,\xi) \le (A(x) \xi, \xi) \quad \text{for all} \ \xi \in \R^d.
\end{align*}
\end{condition}
\begin{condition} \label{CondContinuity}
It holds $A \in C(\Omclo;\R^{d \times d})$, $\varrho \in C(\Omclo)$ and $\varrho > 0$, $d.x.$-a.e.
\end{condition}

From \cite[Theo.~1.5]{BG13} we get:
\begin{theorem} \label{TheoIntroClosable}
Assume Conditions \ref{Cond1Matrix} and \ref{CondContinuity}. Then the form $(\mathcal E,\mathcal D)$ is closable with closure denoted by $(\mathcal E,D(\mathcal E))$. The closure is a strongly local, regular Dirichlet form.
\end{theorem}

Next we fix the regularity and differentiability conditions on $\varrho$.
\begin{condition} \label{CondDiffDensity}
For the density it holds $\sqrt{\varrho} \in H^{1,2}_{\text{loc}}(\Omclo)$. There exists $p \ge 2$ with $p > \frac{d}{2}$ such that 
\begin{align}
\frac{|\nabla \varrho|}{\varrho} \in L_{\text{loc}}^p(\Omclo \cap \{ \varrho > 0 \},\mu). \label{EqCondDiffDensity}
\end{align}
\end{condition}

For the pointwise construction of the process in \cite{BG13} additional regularity on the boundary and the matrix is assumed.
\begin{condition} \label{Cond1DiffBoundary}
There exists a subset $\Gamma_2 \subset \partial \Om$, open in $\partial \Om$, such that the boundary is locally $C^2$-smooth at every $x \in \Gamma_2$ and $\capE(\partial \Om \setminus \Gamma_2)=0$. For the matrix $A$ it holds $A \in C^1(\Omclo)$. 
\end{condition}
Here $\capE$ denotes the capacity of the Dirichlet form $(\mathcal E,D(\mathcal E))$. From general Dirichlet form theory it follows that $(\mathcal E,D(\mathcal E))$ has an associated $L^2$-strongly continuous contraction semigroup $(T^2_t)_{t \ge 0}$ and resolvent $(G^2_\lambda)_{\lambda > 0}$. Furthermore, there exists an associated $L^p$-strongly continuous contraction semigroup $(T^p_t)_{t \ge 0}$ and resolvent $(G^p_\lambda)_{\lambda > 0}$ for $p \in [1,\infty)$. With \emph{associated} we mean here that $T^p_t f = T^2_t f$ for all $t \ge 0$ and $f \in L^1(\Omclo,\mu) \cap L^\infty(\Omclo,\mu)$. In the same way we define \emph{associated} for the $L^p$-resolvent. 
We denote the corresponding infinitesimal generator by $(L_p,D(L_p))$ and call it shortly the $L^p$-generator.
It is important to note that $(T^p_t)_{t \ge 0}$ is analytic for $1 < p < \infty$, see the explanation in \cite{BGS13} before Theorem 1.4 therein.
From \cite[Theo.~1.1.4]{BG13} we get.

\begin{theorem} \label{TheoDiffProcess}
Assume Conditions \ref{Cond1Matrix}, \ref{CondContinuity}, \ref{CondDiffDensity} and \ref{Cond1DiffBoundary}.  
Let $p$ be as in Condition \ref{CondDiffDensity}, $\Gamma_2$ as in Condition \ref{Cond1DiffBoundary}. Define $E_1 := (\Om \cup \Gamma_2) \cap \{ \varrho > 0\}$. Then there exists a diffusion process (i.e.,~a strong Markov process having continuous sample paths) 
\begin{align*}
\mathbf M = (\mathbf \Om,\mathcal F, (\mathcal F_t)_{t \ge 0}, (\mathbf X_t)_{t \ge  0},(\mathbb P_x)_{x \in E \cup \{ \Delta \}})
\end{align*}
with state space $\Omclo$ and cemetery $\Delta$, the Alexandrov point of $\Omclo$. The set $E_1$ is $\mathbf M$-invariant (in the sense of Definition \ref{DefRestrictedHuntProcess}). The transition semigroup $(P_t)_{t \ge 0}$ is associated with $(T^p_t)_{t \ge 0}$ and is $\mathcal L^p$-strong Feller, i.e.,~$P_t \mathcal L^p(\Omclo,\mu) \subset C(E_1)$ for $t>0$. The corresponding resolvent kernels $(R_\alpha)_{\alpha > 0}$ are also $\mathcal L^p$-strong Feller on $E_1$.
The process has continuous paths on $[0,\infty)$ and it solves the martingale problem associated with $(L_p,D(L_p))$ for starting points in $E_1$, i.e.,
\begin{align*}
M_t^{[u]}:=\widetilde u({\mathbf X}_t) - \widetilde u(x) - \int_0^t{L_p u({\mathbf X}_s)~ds},~t\geq0,
\end{align*}
is an $(\mathcal{F}_t)$-martingale under $\mathbb{P}_x$ for all $u \in D(L_p)$, $x \in E_1$. 
\end{theorem}
Here $\widetilde{u}$ denotes the continuous version of $u$ on $E_1$ that is provided by the regularity of $D(L_p)$ see \cite[Theo.~1.11]{BG13}.
Here $(P_t)_{t \ge 0}$ being associated with $(T^p_t)_{t \ge 0}$ means that $P_t f$ is a $\mu$-version of $T^p_t f$ for $f \in \mathcal L^1(\Omclo,\mu) \cap \mathcal B_b(\Omclo)$ (the space of Borel measurable bounded functions). By $\mathcal L^p(\Omclo,\mu)$ we denote the space of all $p$-integrable functions on $(\Omclo,\mu)$.

In \cite{BG13} it is shown that the set
\begin{align*}
\mathcal D_{\text{Neu}} := \{ u \in C^2_c((\Om \cup \Gamma_2) \cap \{ \varrho > 0 \}) \, | \, (\unitnormal, A \nabla u) = 0 \ \text{on} \, \Gamma_2 \}
\end{align*}
is contained in the domain $D(L_p)$, the domain of the $L^p$-generator for $p$ as in Condition \ref{CondDiffDensity}. 
Here $\unitnormal(x)$ denotes the outward unit normal at $x \in \Gamma_2 \subset \partial \Om$ orthogonal to $\partial \Om$. 
By $C^2((\Om \cup \Gamma_2) \cap \{ \varrho > 0 \})$ we denote the space of all $C^2$-smooth functions on $\Om \cap \{ \varrho > 0 \}$, such that the functions and their derivatives admit continuous extensions to the boundary part $\Gamma_2 \cap \{ \varrho > 0 \}$. The subindex $c$ marks that the functions have compact support and that the support is contained in $(\Om \cup \Gamma_2) \cap \{ \varrho > 0 \}$, i.e., $\supp[u] \subset \subset (\Om \cup \Gamma_2) \cap \{ \varrho > 0 \}$ for $u \in C^2_c((\Om \cup \Gamma_2) \cap \{ \varrho > 0 \})$.

Using partial integration we get that for $u \in \DNeu$ the $L^p$-generator has the following form:
\begin{align}
L_p u = \Ltil u := \sum^d_{i,j=1} a_{ij} \pdi \pdj u + \sum^d_{j=1} \left(\sum^d_{i=1} \pdi a_{ij} + \sum^d_{i=1} \frac{1}{\varrho} a_{ji} \pdi \varrho \right) \pdj u. \label{EqLtil}
\end{align}

So we have that
\begin{align*}
M_t^{[u]}:=u({\mathbf X}_t) - u(\mathbf X_0) - \int_0^t{ \hat{L} u({\mathbf X}_s)~ds},~t\geq0,
\end{align*}
is an $(\mathcal{F}_t)$-martingale under $\mathbb{P}_x$ for all $x \in E_1$ and $u \in \DNeu$.

We aim to extend the martingale solution property to a larger class of functions, namely $C^2_c(E_1)$. If we do a partial integration for functions in that larger space, there appears an additional boundary term. In order to incorporate this additional term in the martingale formulation we need to construct the boundary local time of $\mathbf M$ at the boundary part $\Gamma_2 \cap \{ \varrho > 0\}$. Roughly speaking the boundary local time measures the time of $(\mathbf X_t)_{t \ge 0}$ at the boundary on a different time scale.

We apply the theory of \cite{FOT11} which gives existence of additive functionals in various classes for processes associated with Dirichlet forms. In particular, the result of \cite[Theo.~5.1.6]{FOT11} allows us to construct additive functionals without exceptional set. Therein it is assumed that the semigroup $(P_t)_{t > 0}$ of the process  is absolutely continuous (w.r.t.~the reference measure) on the whole state space of the process. In our setting the semigroup is absolutely continuous on the subset $E_1 = (\Om \cup \Gamma_2) \cap \{ \varrho > 0 \}$ only. However, by considering the \emph{restriction} of $\mathbf M$ to $E_1$ we get a semigroup that fulfills the absolute continuity condition for \emph{every point} in the state space. The reader is referred to Definition \ref{DefRestrictedHuntProcess} below for the definition of a restricted process. From Theorem \ref{TheoDiffProcess} we can conclude that the restricted process enjoys the same properties as the original one.
\begin{corollary} \label{CoroRestriction}
Assume the same conditions as in Theorem \ref{TheoDiffProcess} and denote by
$\mathbf M$ the diffusion process constructed in Theorem \ref{TheoDiffProcess}. Define the restricted process 
\begin{align*}
\mathbf M^1 := (\mathbf \Om^1,\mathcal F^1, (\mathcal F^1_t)_{t \ge 0}, (\mathbf X^1_t)_{t \ge  0},(\mathbb P^1_x)_{x \in E_1 \cup \{ \Delta \}})
\end{align*}
of $\mathbf M$ to $E_1$ (see Definition \ref{DefRestrictedHuntProcess}). Then $\mathbf M^1$ is a $\mathcal L^p$-strong Feller diffusion process with state space $E_1$ and cemetery $\Delta$. The transition semigroup $(P^1_t)_{t \ge 0}$ of $\mathbf M^1$ is $\mathcal L^p$-strong Feller, i.e.,~$P^1_t \mathcal L^p(\Omclo,\mu) \subset C(E_1)$ for $t>0$. In particular, $(P^1_t)_{t > 0}$ is absolutely continuous on $E_1$.
\end{corollary}
This corollary follows by a general theorem on Hunt processes, see e.g. \cite[Appendix, (A.2.23)]{FOT11} below. The continuity of the sample paths and the $\mathcal L^p$-strong Feller property follows from the fact that $(\mathbb P^1_x)_{x \in E_1 \cup \{ \Delta \}}$ is obtained as restriction of the original path measure $(\mathbb P_x)_{x \in E \cup \{ \Delta \}}$ to $\mathcal F^1$.
To avoid overloading of notation we denote the path functions of $\mathbf M^1$ just by $(\mathbf X_t)_{t \ge  0}$ (instead of $(\mathbf X^1_t)_{t \ge  0}$). In the same way we denote the path measure just by $(\mathbb P_x)_{x \in E_1 \cup \{ \Delta \}}$.

For the restricted process $\mathbf M^1$ we construct the local time at compact boundary parts first. For this we have to check regularity of potentials of surface measures at compact boundary parts in $\Gamma_2 \cap \{ \varrho > 0\}$. See Theorem \ref{TheoFiniteEnergyIntegral} and Corollary \ref{LemmaLocaltimeCompactBoundary} below. As mentioned above our strategy is to apply our regularity results (see \cite{BG13}) for elliptic PDE. For this we identify the potentials as weak solutions to elliptic PDE with sufficiently regular right-hand side. The regularity of the right hand side (i.e., integration of a function w.r.t.~to the surface measure) is shown using Sobolev space theorems.

Using a localization procedure as in \cite{FOT11} we obtain the existence of the local time $(\ell_t)_{t \ge 0}$ at $\Gamma_2 \cap \{ \varrho > 0\}$. This local time is in Revuz correspondence to the restricted surface measure $1_{\Gamma_2 \cap \{ \varrho > 0 \}} \sigma$.
Note that in general it is not possible to construct the local time at the whole boundary because $\sigma$ need not to be a smooth measure (see Definition \ref{DefFiniteEnergy} below).

Using the local time we can characterize the process $(u(\mathbf X_t))_{t \ge 0}$ for $u \in C^2_c(E_1)$. More precisely,
\begin{align*}
M_t^{[u]}:= u({\mathbf X}_t) - u(\mathbf X_0) - \int_0^t{ \Ltil u({\mathbf X}_s)~ds} + \int_0^t (A \nabla u, \unitnormal) \, \varrho \, (\mathbf X_s) \, d \ell_s,~t\geq0,
\end{align*}
is an $(\mathcal{F}^1_t)$-martingale under $\mathbb{P}_x$ for all $x \in E_1$ and $u \in C^2_c(E_1)$, see Theorem \ref{TheoMartingaleWithLocalTime} below.

We can characterize the quadratic variation process of the martingale $(M^{[u]}_t)_{t \ge 0}$ \linebreak in terms of the matrix coefficient, see Theorem \ref{TheoremMartingaleEnergyMeasure} below. Altogether, we get a semimartingale decomposition for $(u(\mathbf X_t)-u(\mathbf X_0))_{t \ge 0}$. 

Using a localization technique we get such a Skorokhod decomposition for the process itself.
Denote by $(b_i)_{1 \le i \le d}$ the first-order coefficients of $\Ltil$, see \eqref{EqLtil}. Then we have for $t \ge 0$
\begin{align*}
\mathbf X_{t \minw \lifetime}^{(i)} - \mathbf X^{(i)}_0 = \int_0^{t \minw \lifetime} b_i(\mathbf X_s) \, ds - \int_0^{t \minw \lifetime} (e_i, A \unitnormal) \varrho \, (\mathbf X_s) \, d \ell_s + M^{(i)}_{t \minw \lifetime} \ \ \mathbb P_x-\text{a.s.}
\end{align*}
for $x \in E_1$ and $1 \le i \le d$. Here $e_i$, $1 \le i \le d$, denotes the $i$-th unit vector. The $(M^{(i)}_t)_{t \ge 0}$, $1 \le i \le d$, are continuous local martingales (up to the lifetime $\lifetime$) with quadratic variation process (up to $\lifetime$)
\begin{align*}
\langle M^{(i)}, M^{(j)} \rangle_{\cdot \minw \lifetime} = 2 \, \int_0^{t \minw \lifetime} a_{ij} \, (\mathbf X_s) \, ds \ \text{for} \ 1 \le i,j \le d. 
\end{align*}
See Theorem \ref{TheoLocalMartingalePx} below. Let us emphasize that these decompositions hold under the path measures $\mathbb P_x$ for every $x \in E_1$, i.e., we have again a pointwise statement. For conservative processes we can further conclude existence of weak solutions. 
\begin{theorem} \label{TheoExWeakSolution}
Let $\sigma : \Omclo \to \R^{d \times r}$ be a mapping of $d \times r$ matrices, $r \in \N$, such that $A := \sigma \sigma^\top$ is a continuous mapping of strictly elliptic symmetric matrices.
Assume that $A$ and $\varrho$ satisfy Condition \ref{Cond1Matrix}, Condition \ref{CondContinuity} and Condition \ref{CondDiffDensity}. Assume additionally that the corresponding Dirichlet form (closure of \eqref{EqGradientDirForm}) is conservative and that $\Gamma_2$ and $A$ satisfy Condition \ref{Cond1DiffBoundary}. 
Let $\mu_0 \in \mathcal P(E_1)$ (probability measures on $E_1$). Endow the path space of $\mathbf M^1$ with the law $\mathbb P_{\mu_0}$ (see the proof below).
Then there exists (possibly on an extension of the probability space of $\mathbf M^1$) an $r$-dimensional Brownian motion $W$ such that
\begin{equation}
\mathbf X_t = \mathbf X_0 + \int_0^t \left(\nabla A + A \frac{\nabla \varrho}{\varrho}\right) \, \, (\mathbf X_s) \, ds - \int_0^t \varrho A \unitnormal \, (\mathbf X_s) d \ell_s \\
+ \int_0^t \sqrt{2} \, \sigma(\mathbf X_s) \, d W_s, \ \ t \ge 0,
\end{equation}
and $\mathcal L(\mathbf X_0) = \mu_0$.
\end{theorem}
For the proof see Section \ref{SecSemiMartingale} below. Here $(\nabla A)_{i} :=\sum^d_{j=1} \pdj a_{ij}(x) $, $1 \le i \le d$.
As admissible starting distributions $\mu_0$ for $\mathbf X_0$ we allow general probability measures on $E_1$, in particular point measures are allowed.
The assumptions on $\varrho$ allow that the drift-term $A \frac{\nabla \varrho}{\varrho}$ has very strong singularities, in particular potentials of Lennard-Jones type can be handled with our method. In the case of a given mapping of strictly elliptic symmetric matrices $A : \Omclo \to \R^{d \times d}$ we may choose $r=d$ and $\sigma = \sqrt{A}$, the unique strictly elliptic square-root of $A$.

Next let us compare our results to other results on construction of (reflected) diffusions.
Chen considers a gradient Dirichlet form with matrix and density with mild differentiability conditions and global lower and upper bounds on the coefficients. Under mild assumptions on the boundary he provides a semimartingale decomposition holding for quasi-every starting point, see \cite{Che93}. Bass and Hsu give a pointwise semimartingale decomposition for reflected Brownian motion in Lipschitz domains, see \cite{BaHs91} and \cite{BaHs90}. They obtain also some results in H\"older domains.


Fukushima and Tomisaki, see \cite{FT95} and \cite{FT96}, construct classical Feller processes associated to gradient Dirichlet forms with uniformly elliptic coefficient matrix. Using the results of \cite{FOT94} a semimartingale decomposition is given. Note that the assumptions in \cite{FT95} and \cite{FT96} exclude singular drifts and the constructed semigroups are classical Feller semigroups.
Hence our results are not covered by the previously mentioned works. 

Pardoux and Williams (see \cite{PaWi94}) as-well as Williams and Zheng (see \cite{WiZh90}) provide approximations of reflected diffusions by diffusions on $\R^d$ or the interior of the state space. The convergence results are obtained by Dirichlet form methods.

Pathwise uniqueness for Brownian motion (without drift) on domains with certain boundary smoothness are obtained by Bass, Burdzy, Chen and Hsu, see \cite{BaHs00}, \cite{BBC05} and \cite{BB08}.

Let us briefly give reference to classical results on reflected diffusions: For strong solutions, see the works of Tanaka (\cite{Tan79}) and of Lions and Sznitman (\cite{LS84}). The latter was generalized to domains with less smooth boundary by Saisho (\cite{Sai87}) and Dupuis and Ishii (\cite{DI93}). Stroock and Varadhan construct reflected diffusions via the (sub-)martingale formulation, see  (\cite{SV71}).

Let us now come to results on SDEs with singular drift.
Trutnau (\cite{Tru03}) constructs a generalized (non-symmetric) Dirichlet form with singular non-symmetric drift term. The corresponding diffusion process is constructed and a Skorokhod decomposition is obtained for quasi-every starting point. Recently we got to know about an article of Shin and Trutnau (\cite{ST14}), which also handles pointwise Skorokhod decompositions for reflected diffusions based on the theory of \cite{FOT11}. The assumed conditions are complementary to ours and the applied methods differ from ours.
The strategy to construct martingale solutions to singular SDEs for explicitly specified starting points has been successfully applied in the already mentioned article by Albeverio, Kondratiev and R\"ockner (\cite{AKR03}), Fattler and Grothaus (\cite{FG07}) and in our own work (\cite{BGS13}, \cite{BG13}). The results of this article apply to all these settings and thus we can show that those solutions even have a Skorokhod decomposition and yield weak solutions.  


For diffusions on $\R^d$, $d \in \N$, there are already some results on SDEs with drifts having singularities. H\"ohnle proves existence of local solutions (\cite{Hoh94}) and criterions for existence of global solutions (\cite{Hoh96}) of Brownian motion distorted by singular drifts.
Krylov and R\"ockner prove existence and uniqueness for strong solutions for SDEs with time-dependent drift terms, see \cite{KrRo05}.

Let us summarize now the main progress of this article:

\begin{sequence}
\item[(i)] We prove regularity and boundedness of potentials of surface measures using elliptic \\ regularity results and Sobolev space theorems.
\item[(ii)] We construct Skorokhod decompositions for reflecting diffusions with variable diffusion \\ coefficients and strongly singular - possibly discontinuous - drift term, that can start \\ from every point in an explicitly known set $E_1 \subset E$.
\item[(iii)] We construct stochastic dynamics for physically reasonable models with hydrodynamic \\ and singular pair interaction.
\end{sequence}


\section{Construction of Strict Additive Functionals}

In this section we present relevant definitions for smooth measures, restriction of processes and additive functionals. Especially, we illustrate how \cite[Theo.~5.1.6]{FOT11} is applied to construct additive functionals on the set $E_1$. Readers who are familiar with \cite[Ch.~5]{FOT11} may skip this section. 
Let $\mathbf M = (\mOm,\mathcal F,(\mathcal F_{t})_{t \ge 0},(\mathbf X_{t})_{t \ge 0},(\mathbb P_{x})_{x \in E})$ be the diffusion process from Theorem \ref{TheoDiffProcess}. 

In \cite[Ch.~5]{FOT11} additive functionals of Hunt processes are constructed on subsets of the state space complemented by an exceptional set. In general this exceptional set is non-empty and depends on the constructed functional.
If the semigroup $(P_t)_{t > 0}$ is absolutely continuous, these results are refined to yield additive functionals with empty exceptional set, see \cite[Theo.~5.1.6]{FOT11}. 

In our case we have that the semigroup $(P_t)_{t > 0}$ is absolutely continuous on $E_1$ only. So we have to apply \cite[Theo.~5.1.6]{FOT11} to the restriction of $\mathbf M$ to $E_1$ rather than the original process $\mathbf M$. 
So altogether we can construct additive functionals for a fixed exceptional set $E \setminus E_1$ that is given in advance.

Let us briefly recall the definition of several classes of measures, see \cite[Ch.~2, Sec.~2]{FOT11}. For the notion of nests and generalized nests, see \cite[Ch.~2, p.~69]{FOT11}. We say that a nest $(F_n)_{n \in \N}$ is \emph{associated}\index{nest!associated with a measure} with a measure $\nu$ if $\nu(F_n) < \infty$ for all $n \in \N$.
\begin{definition} \label{DefFiniteEnergy}
A positive Borel measure $\nu$ is called \emph{smooth} if it charges no set of capacity zero and has an associated generalized nest, see \cite[p.~83]{FOT11}. The class of all smooth measures is denoted by $S$.

We denote by $S_0$ the class of all positive Radon measures of \emph{finite energy integrals}, i.e., for $\nu \in S_0$ the mapping
\begin{align*}
D(\mathcal E) \cap C_c(E) \ni f \mapsto \int_E f d \nu \, \in \R
\end{align*}
is continuous in the $\mathcal E_1$-norm ($\ \Vert \cdot \Vert_{\mathcal E_1} := ( (\cdot,\cdot)_{L^2} + \mathcal E(\cdot, \cdot) \ )^{1/2}$), see \cite[Ch.~2, p.~79]{FOT11}. 

If $\nu \in S_0$ then for $0 < \alpha < \infty$ there exists a corresponding unique $\alpha$-potential $U_\alpha \nu \in D(\mathcal E)$ such that
\begin{align*}
\mathcal E_\alpha(U_\alpha \nu,f) = \int_E f d \nu \quad \text{for all} \ f \in D(\mathcal E),
\end{align*}
see \cite[Ch.~2, Sec.~2]{FOT11}.

Define (as in \cite[Ch.~2, p.~81]{FOT11})
\begin{align*}
S_{00} = \left \{ \nu \in S_0 \, \bigg | \, \nu(E) < \infty, \, U_\alpha \nu \, \text{is essentially} \, \mu-\text{bounded for} \, \alpha > 0 \right \}.
\end{align*}
\end{definition}
Let us introduce the class of \emph{smooth measures in the strict sense}, see \cite[Ch.~5, p.~238]{FOT11}.
\begin{definition} \label{DefSmoothStrictSense}
Let $\mathbf M$ be the Hunt process as introduced in the beginning of the section. For $F \subset E$ a Borel set we denote by $\sigma_{F}$ the hitting time of $F$.
We say that a positive Borel measure $\nu$ on $E$ is \emph{smooth in the strict sense} if there exists a sequence $(E_n)_{n \in \N}$ of Borel sets increasing to $E$ such that $1_{E_n} \cdot \nu \in S_{00}$ for each $n \in \N$ and
\begin{align*}
\mathbb P_x \left( \limn \sigma_{E \setminus E_n} \ge \lifetime \right) = 1, \, \text{for all} \, x \in E.
\end{align*} 
The class of all smooth measures in the strict sense is denoted by $S_1$.
\end{definition}

Next, let us introduce the notion of the \emph{restriction of a process} to subsets of $E^\Delta$.
\begin{definition} \label{DefRestrictedHuntProcess}
Let $\subE \subset E$ be nearly Borel. Define 
\begin{align*}
\mOm^{\tilde{E}} = \left \{ \omega \in \mOm \  \big| \ \mathbf X_t(\omega) \in \subE^\Delta, \, \mathbf X_{t-}(\omega) \in \subE^\Delta \ \text{for all} \ t \ge 0  \right \}.
\end{align*}
We say that $\subE$ is $\mathbf M$-invariant if $\mathbb P_x(\mOm^{\tilde{E}}) = 1 \ \text{for all} \ x \in \subE^\Delta$.
If $\subE$ is $\mathbf M$-invariant we define the \emph{restriction} of $\mathbf M$ to $\subE$ by
\begin{align*}
\mMRes := (\mOmRes,\mFRes,(\mFtRes{t})_{t \ge 0},(\mXtRes{t})_{t \ge 0},(\mPxRes{x})_{x \in \subE}),
\end{align*}
with $\mOmRes := \mOm^{\tilde{E}}$, $\mFRes = \mathcal F \cap \mOmRes $, $\mXtRes{t} = {\mathbf X}_{t}|_{\mOmRes} \ \text{for} \ t \ge 0 $,
$\mPxRes{x} = \mathbb P_x|_{\mFRes} \ \text{for} \ x \in \subE $.
As filtration $(\mFtRes{t})_{t \ge 0}$ we take the minimum completed admissible filtration of $(\mXtRes{t})_{t \ge 0}$. For $\widetilde{\mathcal F}^0_\infty := \bigcup_{t > 0} \mFtRes{t}$ define $\mFResInftCompl := \bigcap_{m \in \mathcal P(E^\Delta_1)} \mFResInftCompl^m$. Here $\mFResInftCompl^m$ denotes the \emph{completion} of $\mFResInftCompl^0$ under $\mathbb P_m|_{\mFRes}$. 
See e.g. \cite[Appendix, p.~386]{FOT11} for further details. 
\end{definition}
See \cite[Appendix, (A.2.23)]{FOT11} for this definition. As in the mentioned reference we get that the restriction of a Hunt process is again a Hunt process.

Let us now introduce the definition of an additive functional in the sense of \cite[p.~222, (A.1) and (A.2)]{FOT11}. We denote by $\shiftop_t: \mOm \mapsto \mOm$, $t \ge 0$, the shift operator: $\shiftop_t \omega \ ( \cdot ) = \omega( \cdot + t )$, i.e., the path is shifted to the left by $t$.
\begin{definition} \label{DefAddFunc}
Let $N \subset E$ be a properly exceptional set (see \cite[Ch.~4, p.~153]{FOT11}) and $\mMRes[E \setminus N]$ be the restriction of the Markov process $\mathbf M$ to the set $E \setminus N$ as in Definition \ref{DefRestrictedHuntProcess}.
Let $\Lambda \in \mFResInftCompl$ with $\mPxRes{x}(\Lambda) = 1$ for $x \in E \setminus N$ and $\shiftop_t \Lambda \subset \Lambda$ for $t \ge 0$.

A mapping $A = (A_t)_{t \ge 0} : \mOmRes[E \setminus N] \to \R_0^+$ is called an \emph{additive functional} (AF) with exceptional set $N$ and defining set $\Lambda$ if $A_0 = 0$, $A_t$ is $\mFtRes[E \setminus N]{t}$-adapted, for $t \ge 0$, and the following properties hold:
\begin{sequence}
\item[(i)] $|A_{t}(\omega)| < \infty$ \ \text{for} \ $t < \lifetime(\omega)$, $\omega \in \Lambda$.
\item[(ii)] For every $\omega \in \Lambda$, $[0,\infty) \ni t \mapsto A_t(\omega)$ is right continuous and has left limit on $[0, \lifetime(\omega))$.
\item[(iii)] $A_{t}(\omega) = A_{\lifetime(\omega)}(\omega)$ \ \text{for} \ $t \ge \lifetime(\omega)$ and $\omega \in \Lambda$.
\item[(iv)] $A_{t+s}(\omega) = A_t(\omega) + A_s(\shiftop_t \omega)$  \text{for} \ $0 \le s,t < \infty$ and $\omega \in \Lambda$.
\end{sequence}

A functional $\funcA$ is called \emph{finite} if (i) holds for $t < \infty$ instead of just $t < \lifetime(\omega)$.

A functional $\funcA$ is called \emph{positive} if $A_{t}(\omega) \in [0,\infty]$ for $t \ge 0$, $\omega \in \Lambda$.

A functional $\funcA$ is called \emph{continuous} if instead of (ii) the stronger condition holds:
\begin{sequence}
\item[(ii')] For every $\omega \in \Lambda$, $[0,\infty) \ni t \mapsto A_t(\omega)$ is continuous.
\end{sequence}

A mapping $A = (A_t)_{t \ge 0} : \mOmRes[E^{\Delta} \setminus N] \to \R \cup \{ \infty \}$ is called a \emph{local continuous additive functional}\index{local CAF|see{local continuous additive functional}} \index{local continuous additive functional} with exceptional set $N$ and defining set $\Lambda$ as above, if $A_0 = 0$, $A_t$ is $\mFtRes[E^{\Delta} \setminus N]{t}$-adapted and:
\begin{sequence}
\item[(i)] $|A_{t}(\omega)| < \infty$ \ \text{for} \ $t < \lifetime(\omega)$, $\omega \in \Lambda$.
\item[(ii)] For every $\omega \in \mOmRes[E^{\Delta} \setminus N]$, $[0,\lifetime(\omega)) \ni t \mapsto A_t(\omega)$ is continuous.
\item[(iii)] $A_{t+s}(\omega) = A_t(\omega) + A_s(\theta_t \omega)$  \text{for} \ $0 \le s,t < \infty$ with $s+t < \lifetime(\omega)$ and $\omega \in \Lambda$.
\end{sequence}
\end{definition}

Observe that both the exceptional set $N \subset E$ and the defining set $\Lambda \subset \mOm$ depend on the additive functional $A$. However, we are interested in the construction of additive functionals, where the exceptional set is given in advance.


Note that we assume the additivity on the set $\Lambda$ and $\Lambda$ is chosen independently of $x$. Such a PCAF is called \emph{perfect} in the sense of \cite[Ch.~IV, Def.~1.3]{BG68}. 

Let us introduce now an important link between measures and AF, the so-called Revuz correspondence. For our purpose the following equivalent characterization of the Revuz correspondence is most suitable: A smooth measure $\nu$ is said to be in \emph{Revuz correspondence} to a PCAF $A$ if
\begin{align}
\int_{E \setminus N} h \,  (U^\alpha_{A} f) \, d \mu = \int_{E \setminus N} (R^{E \setminus N}_\alpha h) \, f d \nu \ \text{for} \ h,f \in \mathcal B^+(E \setminus N) \ \text{and} \ \alpha > 0 \label{EqRevuzCorrespondence}
\end{align}
with $U^\alpha_{A} f(x) := \mathbb E^{(E \setminus N)}_x[\int_0^\infty e^{-\alpha s} f(\mathbf X_s) d A_s]$, $x \in E^\Delta \setminus N$, $N$ being the exceptional set of $A$. Here $(R^{E \setminus N}_\alpha)_{\alpha > 0}$ and $\mathbb E_{\cdot}^{E \setminus N} [\, \cdot \,]$ denote the resolvent and expectation, respectively, of the restricted process $\mMRes[E \setminus N]$.  

See \cite[Theo.~5.1.3]{FOT11} for these definitions and further equivalent descriptions of the Revuz correspondence.
In the case that $\nu$ has an $\alpha$-potential, i.e., $\nu \in S_0$, the Revuz correspondence is equivalent to:
\begin{align*}
U^1_{A} 1 \ \text{is a $\mathcal E$-quasi-continuous version of} \ U_1 \nu.
\end{align*}
For the definition of $\mathcal E$-quasi-continuous, see \cite[Ch.~2, p.~69]{FOT11}.

We are interested in the construction of additive functionals on the fixed (invariant) set $E_1 \subset E$. Denote the restriction of $\mathbf M$ to $E_1 \cup \{\Delta\}$ by \\ $\mathbf M^1 := (\mathbf \Om^1,\mathcal F^1,(\mathcal F^1_t)_{t \ge 0},(\mathbf X^1_t)_{t \ge 0},(\mathbb P^1_x)_{x \in E_1 \cup \{ \Delta \}})$. The corresponding resolvent and semigroup we denote just by $(R_\alpha)_{\alpha > 0}$ and $(P_t)_{t \ge 0}$, respectively.
The $\mathcal L^p$-strong Feller property yields that the semigroup $(P_t)_{t > 0}$ is absolutely continuous on $E_1$, i.e.,
there exists a $\mathcal B(E_1 \times E_1)$-measurable density $p_t :E_1 \times E_1 \to \R_0^+$, $t > 0$, such that
\begin{align*}
P_t f(x) = \int_{E_1} f(y) \, p_t(x,y) d \mu(y) \ \text{for} \ f \in \mathcal B^+(E_1) \ \text{and} \ x \in E_1.
\end{align*}
According to \cite[Ch.~2.2]{FOT11} we call a function $u \in L^2(E,\mu)$ $\alpha$-excessive\index{excessive} if
\begin{align}
u(x) \ge 0 \ \text{and} \  e^{-\alpha t} T_t u \, (x) \le u \, (x) \quad \ \text{for} \ \mu\text{-a.e.}\,x. \label{EqAlphaExcessive}
\end{align}

The absolute continuity condition of the semigroup transfers to the resolvent. From \cite[Lem.~4.2.4]{FOT11} it follows that the resolvent $(R_\alpha)_{\alpha > 0}$ has a non-negative density $(r_\alpha(x,y))_{\alpha > 0}$, $x,y \in E_1$, that is $\alpha$-excessive both in $x$ and $y$.
For $\nu \in S_{0}$ the corresponding $\alpha$-potential $U_\alpha \nu$ has a quasi-continuous and $\alpha$-excessive version $\widetilde{U_\alpha} \nu$ that is obtained by the resolvent density:
\begin{align*}
\widetilde{U_\alpha} \nu \, (x) = R_\alpha \nu \, (x) := \int_{E_1} r_\alpha (x,y) d \nu(y), \ \text{for every} \ x \in E_1,
\end{align*}
see \cite[Ex.~4.2.2]{FOT11}. For $\nu \in S_{00}$ the function $R_\alpha \nu$ is even bounded for \emph{every} $x \in E_1$ and \eqref{EqAlphaExcessive} holds for every $x \in E_1$. 


This potential is then used to construct a PCAF for $\nu \in S_{00}$. So by applying \cite[Theo.~5.1.6]{FOT11} to the restricted process $\mathbf M^1$ we obtain the following theorem.
\begin{theorem} \label{TheoremConstrTildeS}
Let $\nu \in S_{00}$ and let $(R_{\alpha} \nu)_{\alpha > 0}$ be the corresponding potentials. 

Then there exists a unique finite PCAF $(\widetilde{A}_t)_{t \ge 0}$ that is in Revuz correspondence to $\nu$ and has the exceptional set $E \setminus E_1$. For $\widetilde{A}_t$ it holds
\begin{align}
\mathbb E^1_x \left[ \int_0^\infty \exp(-\alpha s) d \widetilde{A}_s \right] = {R_\alpha \nu} \, (x) \quad \text{for every} \ x \in E_1 \ \text{and} \ \alpha > 0. \label{EqPotentialTilde}
\end{align}
\end{theorem}

Let us call a PCAF with exceptional set $E \setminus E_1$ from now on \emph{strict on} $E_1$. For a strict finite PCAF $(A_t)_{t \ge 0}$ with Revuz measure $\nu$ and $f \in \mathcal B^+_b(E_1)$ the mapping $[0,\infty) \ni t \mapsto \int_0^t f(\mathbf X_s) d A_s$ defines again a strict finite PCAF. The next lemma shows that the corresponding measure is given by $f \, \nu$, i.e., multiplication with a Borel bounded function is compatible with the Revuz correspondence. Denote by $\supp[\nu]$, $\nu$ a measure, the topological support of a measure $\nu$.

\begin{lemma} \label{LemmadfAt}
Let $(A_t)_{t \ge 0}$ be a strict finite PCAF (on $E_1$) with Revuz measure $\nu \in S_{00}$. Let $M \in \mathcal B(E_1)$ such that $\supp[\nu] \subset M$ and $f \in \mathcal B_b(M)$. Then the mapping $f \cdot A := ((f \cdot A)_t)_{t \ge 0}$, 
\begin{align*}
(f \cdot A)_t = \int_0^t f(X_s) d A_s, \ t \ge 0,
\end{align*}
defines a strict finite CAF with same defining set as $A$.

It holds $\mathbb E_x[|(f \cdot A)_t|] < \infty$ and $\mathbb E_x[(f \cdot A)^2_t] < \infty$ for $0 \le t < \infty$ and every $x \in E_1$.
We have
\begin{align}
P_s \mathbb E_\cdot \left[ \int_0^t f(\mathbf X_r) \, d A_r \right] (x) \overset{s \to 0}{\longrightarrow}  \mathbb E_x \left[ \int_0^t f(\mathbf X_r) \, d A_r \right] \quad \text{for every} \ x \in E_1. \label{EqConvergenceS}
\end{align}
If $f \in \mathcal B^+_b(M)$ and $f \, \nu \in S_{00}$, we have for $\alpha > 0$
\begin{align}
\mathbb E_x \left[ \int_0^\infty e^{-\alpha s} f(X_s) d A_s \right] = R_\alpha f \, \nu (x) \quad \text{for every} \ x \in E_1. \label{EqUftilde}
\end{align}
\end{lemma}
For a proof of this lemma, see \cite[Lem.~6.1.15]{Ba14}.

\section{Construction of the Local Time and the Martingale Problem for $C_c^2$-functions}

Using the theory of \cite[Ch.~5]{FOT11} as presented in the previous section, we construct a boundary local time at $\Gamma_2 \cap \{ \varrho > 0 \}$. The local time is constructed as a strict PCAF on $E_1$ which grows only when the process is at $\Gamma_2 \cap \{ \varrho > 0 \}$, see Remark \ref{RemLocaltimeGrows} below. Note that there might be several functionals that have these property. So the term local time does not refer to a specific functional. Nevertheless, we call the functional that we construct \emph{the boundary local time}. 

Throughout this section we keep the same setting as in the introduction. In particular, we assume Conditions \ref{Cond1Matrix}, \ref{CondContinuity}, \ref{CondDiffDensity} and \ref{Cond1DiffBoundary}.
So let us now fix the process $\mathbf M^1$ from Corollary \ref{CoroRestriction} obtained as the restriction of $\mathbf M$ from Theorem \ref{TheoDiffProcess} to $E_1$. 


We use the local time as a building block for a Skorokhod decomposition of a sufficiently large class of functions. For our purpose the set $C^2_c(E_1)$ (recall: $E_1 = (\Om \cup \Gamma_2) \cap \{ \varrho > 0 \}$) is large enough since we can locally approximate the coordinate functions $x^{(i)}$, $1 \le i \le d$, by functions in $C^2_c(E_1)$.
In \cite[Ch.~5]{FOT11} an extended semimartingale decomposition (in the mean-while also called \emph{Fukushima decomposition}\index{Fukushima decomposition}) for functions in $D(\mathcal E)$ is given. This decomposition is given in terms of additive functionals. They have properties that naturally generalize the properties of the corresponding objects in the classical semimartingale decomposition to the $\mathcal E$-quasi-everywhere setting in Dirichlet forms. More precisely, for the $\mathcal E$-quasi-continuous version $\widetilde{u}$ of $u \in D(\mathcal E)$ it holds
\begin{align*}
\widetilde{u}(\mathbf X_t) - \widetilde{u}(\mathbf X_0) = N^{[u]}_t + M^{[u]}_t, \quad t \ge 0,
\end{align*}
where $N^{[u]}$ and $M^{[u]}$ are finite CAFs (not necessarily strict) having certain properties for $\mathcal E$-quasi-every point, i.e., except for a set of capacity zero. In particular, $M^{[u]}$ is a square-integrable martingale under $\mathbb P_x$ for $\mathcal E$-quasi-every starting point. The process $N^{[u]}$ is of zero energy, see \eqref{DefNc} below.
Under additional assumptions on $u$ these results are refined to pointwise statements there, in particular the martingale property holds for every point.

We apply \cite[Theo.~5.2.4]{FOT11} and \cite[Theo.~5.2.3]{FOT11} to identify $N^{[u]}$ and $M^{[u]}$, respectively, $u \in C^2_c(E_1)$. Using methods of \cite[Theo.~5.2.5]{FOT11} combined with an additional analysis we deduce a \textit{pointwise} Skorokhod decomposition for $u \in C^2_c(E_1)$, formulated as a classical semimartingale decomposition.

The process $N^{[u]}$ contains an integral w.r.t.~the deterministic time scale $t$ and an integral w.r.t.~the local time. The latter shows then the reflection at the boundary. Due to the singular drift terms we have to take special care of integrability issues, these are solved using the $\mathcal L^p$-strong Feller property of the resolvent, see e.g.~Theorem \ref{TheoremLpIntegraldt} below.

In order to construct the local time at $\Gamma_2 \cap \{ \varrho > 0 \}$ we need a suitable generalized nest of compact sets. This nest will be also used later in the localization technique to prove existence of weak solutions.
Define
\begin{align}
U_n := \left \{ \varrho > \frac{1}{n} \right \} \cap \left \{ x \in \Omclo \, | \, \dist(x,\partial \Om \setminus \Gamma_2) > \frac{1}{n} \right \} \cap B_n(0), n \in \N, \label{EqDefUnLocaltime}
\end{align}
and $K_n := \overline{U_n}$. Here $B_n(0) \subset \R^d$ denotes the open ball of radius $n$ around $0$.
Then $K_n$ is compact and $U_n \subset K_n \subset U_{n+1}$. Since $\Gamma_2$ is assumed to be open in $\partial \Om$, we get
\begin{align*}
E_1 = \bigcup_{n \in \N} U_n = \bigcup_{n \in \N} K_n.
\end{align*}
Since $(U_n)_{n \in \N}$ increases to $E_1$, we have
\begin{align*}
\inf_{n \in \N} \capE(K \setminus U_n) = \capE \bigg( K \setminus \bigcup_{n \in \N} U_n \bigg) = 0
\end{align*}
for every compact set $K \subset \Omclo$ by \cite[Theo.~2.1.1]{FOT11}.
Hence also $\limn \capE(K \setminus K_n)=0$.
So $(K_n)_{n \in \N}$ is a generalized compact nest and associated with $\mu$ since $\mu(K_n) < \infty$ for all $n \in \N$.
Using right-continuity of $(\mathbf X_t)_{t \ge 0}$ at $t=0$ and a similar argument as in the proof of \cite[Lem.~3.7]{BGS13} we get
\begin{align*}
\mathbb P_x \left (\limn \tau_n \ge \lifetime \right) = 1 \ \text{for every} \ x \in E_1
\end{align*}
with $\tau_n$ being the exit time of $K_n$, $n \in \N$.

The results of the last section give that for every measure in $S_{00}$ we get a unique strict finite PCAF. We apply this result to construct strict finite PCAF corresponding to $\sigma_n := 1_{K_n} \, \sigma$, $n \in \N$. 

So we have to show that these measures are of finite energy and that the corresponding $\alpha$-potential is essentially bounded. As mentioned in the introduction we follow the strategy to identify the $\alpha$-potential (locally) as the weak solution to an elliptic PDE with sufficiently regular right-hand-side. In the proof Sobolev space theorems and our elliptic regularity results (see \cite{BG13}) play a crucial role. Our results yield in fact that the potential has a continuous bounded version on $E_1$. 

Since we consider later also measures of the form $f \, 1_{K_N} \, \sigma$, $N \in \N$, with $f \in \mathcal B_b^+(E_1)$, we formulate a more general theorem.
\begin{theorem} \label{TheoFiniteEnergyIntegral}
Let $N \in \N$, $f \in \mathcal B^+_b(E_1)$. Then the measure $f \, \sigma_N (= f \, 1_{K_N} \, \sigma)$ is of finite energy, see Definition \ref{DefFiniteEnergy}.
For $\alpha > 0$ the corresponding potential $U_\alpha f \, \sigma_N$ has a continuous bounded version $\widetilde{U_\alpha f \, \sigma_N}$ on $E_1$, in particular $U_\alpha f \, \sigma_N$ is essentially bounded. 
Hence $f \, \sigma_N \in S_{00}$.
\end{theorem}
\begin{proof}
Let $N \in \N$ and $f \in \mathcal B^+_b(E_1)$.
By \cite[Lem.~4.1(ii)]{BG13} the restriction map $\iota : D(\mathcal E) \to H^{1,2}(U_{N+1} \cap \Om)$, $ v \mapsto v|_{U_{N+1} \cap \Om}$ is well-defined and continuous. So for $v \in D(\mathcal E) \cap C_c(E)$, it holds $v \in H^{1,2}(U_{N+1} \cap \Om)$ and $Tr(v) \in L^2(\GammaN, \sigma_N)$ with $\GammaN := \partial \Om \cap K_N$ and $Tr: H^{1,2}(U_{N+1} \cap \Om) \to L^2(\GammaN,\sigma_N)$ the trace operator, see e.g.~\cite[Ch.~V, Theo.~5.22]{AD75}.
We have the estimate
\begin{align*}
\Vert Tr(v) \Vert_{L^2(\GammaN, \sigma_N)} \le K_1 \Vert v \Vert_{H^{1,2}(U_{N+1} \cap \Om)} \le K_2 \Vert v \Vert_{\mathcal E_1}
\end{align*}
for constants $K_1, K_2 < \infty$.
Because $v \in C_c(E)$ it holds $Tr(v) = v$ $\sigma$-a.e.
Thus
\begin{align*}
\int_{\GammaN} |v| \ f d \sigma_N = \int_{\GammaN} |Tr(v)| \ f d \sigma_N \le \sigma(\GammaN)^{1/2} \Vert f \Vert_{\sup} \Vert Tr(v) \Vert_{L^2(\GammaN,\sigma_N)} \\
\le K_2 \sigma(\GammaN)^{1/2} \Vert f \Vert_{\sup} \Vert v \Vert_{\mathcal E_1}.
\end{align*}

So $f \, \sigma_N$ is of finite energy with $\alpha$-potential $U_\alpha f \, \sigma_N \in D(\mathcal E)$ for $0 < \alpha < \infty$. Thus $U_\alpha f \, \sigma_N \in H^{1,2}(U_k \cap \Om)$ for every $k \in \N$.

Next we show that $U_\alpha f \, \sigma_N$ has a bounded continuous version on $E_1$. Fix $0 < \alpha < \infty$. Let $k \in \N$ with $k \ge N+1$.
Since $Tr : H^{1,q^*}(U_k \cap \Om) \to L^{q^*}(\GammaN,\sigma_N)$ is continuous for $ 1 \le q^* < \infty$, the mapping $T: H^{1,q^*}(U_k \cap \Om) \to \R$, $v \mapsto \int_{\GammaN} Tr(v) \ f d \sigma_N$ is a continuous linear functional.
Choose $2 \le d < q < \infty$ such that $q' := \frac{q}{q-1} > 1$ and $q' < \infty$. We have
\begin{multline*}
\hspace{10pt} \alpha \int_{U_k} U_\alpha f \, \sigma_N \, v \, d \mu + \int_{U_k} (A \nabla U_\alpha f \, \sigma_N,\nabla v) \, d \mu =
\mathcal E_\alpha(U_\alpha f \, \sigma_N,v) \\
= \int_{\GammaN} v f d \sigma_N = T(v) \quad \text{for all} \ v \in C^1_c(U_k). \hspace{20pt}
\end{multline*}
Since $T \in (H^{1,q'}(U_k \cap \Om))'$, there exist $g \in L^q(U_k,dx)$ and a vector-valued mapping $(e_i)_{1 \le i \le d} \in (L^q)^d(U_k,dx)$ such that for all $v \in H^{1,q'}(U_k \cap \Om)$
\begin{align*}
\int_{\GammaN} Tr(v) f d \sigma_N = T(v) = \int_{U_k} g \, v \, dx + \int_{U_k} \sum^d_{i=1} e_i \, \pdi v \, dx.
\end{align*}
Moreover, $\Vert g \Vert_{L^{q}(U_k,dx)} + \Vert e \Vert_{(L^q)^d(U_k,dx)} \le \tilde{C} \Vert T \Vert_{(H^{1,q'}(U_k \cap \Om))'}$ for some $\tilde{C} < \infty$, see \cite[Theo.~3.8]{AD75}.
So $U_\alpha f \, \sigma_N$ solves
\begin{equation*}
 \alpha \int_{U_k} \hspace{-3pt}  U_\alpha f \, \sigma_N \, v \, d \mu + \int_{U_k}\hspace{-2pt}  (A \nabla U_\alpha f \, \sigma_N,\nabla v) \, d \mu 
=  \int_{U_k} \hspace{-5pt} g \, v \, dx +  \int_{U_k} \sum_{i=1}^d e_i \, \pdi v \, dx \hspace{-5pt} \quad \text{for all} \, v \in  C^1_c(U_k).
\end{equation*}
Choose $0 < r < \infty$ in the following way: If $x \in U_k \cap \Om$, choose $r$ such that $B_r(x) \subset U_k \cap \Om$. If $x \in U_k \cap \partial \Om$, choose $r$ such that $B_r(x) \cap \overline{\Om} \subset U_k$ and $B_r(x) \cap \partial \Om \subset \Gamma_2$. Note that in both cases $C^1_c(B_r(x) \cap \Omclo)$ embeds into $C^1_c(U_k)$.
Applying \cite[Theo.~4.4]{BG13} with $p=q$ we get $0 < r' < r$ such that $U_\alpha f \, \sigma_N \in H^{1,q}(B_{r'}(x) \cap \Om)$. 
Choosing $0 < r_1 < r'$ we get by Sobolev embedding that $U_\alpha f \, \sigma_N$ has a continuous bounded version $\widetilde{U_\alpha f \, \sigma_N}$ on $B_{r_1}(x) \cap \Omclo$. 

Since $K_n \subset U_{N+1 \maxw n+1}$ and $K_n$ is compact, $U_\alpha f \, \sigma_N$ has a bounded continuous version on every $K_n$, $n \in \N$. Thus there exists a continuous version $\widetilde{U_\alpha f \, \sigma_N}$ on $E_1$. 

However, the function is only locally bounded. To prove the global boundedness we apply a weak maximum principle, see \cite[Lem.~2.2.4]{FOT11}. 

Choose $n > N$, let $M_n := \sup_{x \in K_n} |\widetilde{U_\alpha f \, \sigma_N \ (x)}| < \infty$. Since $\supp[f \, \sigma_N] \subset \supp[\sigma_N] \subset \GammaN \subset K_n$, it holds $|\widetilde{U_\alpha f \, \sigma_N(x)}| \le M_n$ $\sigma_N$-a.e. By the weak maximum principle we get $\widetilde{U_\alpha \sigma_N} \le M_n$ $\mu$-a.e.~hence by continuity everywhere on $E_1$. Thus $\widetilde{U_\alpha \sigma_N}$ is bounded on $E_1$ and $\Vert U_\alpha f \, \sigma_N \Vert_{L^\infty(E,\mu)} < \infty$.
\flushright
\end{proof}

Applying Theorem \ref{TheoFiniteEnergyIntegral} with $f=1$ we get by Theorem \ref{TheoremConstrTildeS} the following corollary.
\begin{corollary} \label{LemmaLocaltimeCompactBoundary}
For each $n \in \N$ there exists a unique strict finite PCAF corresponding to $1_{K_n} \, \sigma$. These we denote by $(\ell^n_t)_{t \ge 0}$, $n \in \N$.
\end{corollary}

We apply the previous results to construct the local time at the boundary.

\begin{theorem} \label{TheoLocaltimeS1}
The restricted surface measure $1_{\Gamma_2} 1_{\{\varrho > 0 \}} \, \sigma$ is smooth in the strict sense. There exists a corresponding strict PCAF denoted by $(\ell_t)_{t \ge 0}$ and called the local time\index{local time} (at $1_{\Gamma_2} 1_{\{ \varrho > 0 \}}$). 
For the defining set $\Lambda$ it holds $\Lambda \subset \Lambda_n$, $n \in \N$, $\Lambda_n$ being the defining sets of $(\ell^n_t)_{t \ge 0}$ from Corollary \ref{LemmaLocaltimeCompactBoundary}. 

Let $f \in \mathcal B^+_b([0,\infty) \times E_1)$ with $\supp[f] \subset [0,\infty) \times K_m$ for some $m \in \N$. Then it holds
\begin{align*}
\int_0^\infty f(s,\mathbf X_s(\omega)) d \ell_s(\omega) = \int_0^\infty f(s,\mathbf X_s(\omega)) d \ell^m_s(\omega) \quad \text{for all} \ \omega \in \Lambda  
\end{align*}
and
\begin{align*}
\int_0^t f(s,\mathbf X_s(\omega)) d \ell_s(\omega) = \int_0^t f(s,\mathbf X_s(\omega)) d \ell^m_s(\omega) \quad \text{for all} \ \omega \in \Lambda. 
\end{align*}
\end{theorem}
\begin{proof}
Set $\nu := 1_{\Gamma_2}  1_{\{ \varrho > 0 \}} \, \sigma$. Since $\nu(K_n) = \sigma_n(K_n) < \infty$, we have that $(K_n)_{n \in \N}$ is a generalized nest associated with $\nu$. Furthermore, $1_{K_n} \, \nu = \sigma_n \in S_{00}$ for all $n \in \N$ by Corollary \ref{LemmaLocaltimeCompactBoundary}. It is left to check that $\nu$ is smooth. Let $A \subset E$ with $\capE(A)=0$. We have
\begin{align*}
\nu(A) = \nu(A \cap \{ \varrho > 0 \} \cap \Gamma_2) = \sup_{n \in \N} \, \nu(K_n \cap A) = \sup_{n \in \N} \, \sigma_n(A)=0
\end{align*}
since $\sigma_n$ is smooth for all $n \in \N$.

So we can apply \cite[Theo.~5.1.7(i)]{FOT11} to construct a corresponding PCAF $(\ell_t)_{t \ge 0}$ on $E_1$. The functional is constructed using the local times $(\ell^n)_{n \in \N}$ at compact boundary parts from Corollary \ref{LemmaLocaltimeCompactBoundary} in the following way:
For $n$, $m \in \N$ with $n > m$ we have that $(1_{K_{m}} \cdot \ell^n )$ is in Revuz correspondence to $\sigma_m$ by Lemma \ref{LemmadfAt}, \eqref{EqUftilde}. Hence by uniqueness we get that $(1_{K_{m}} \cdot \ell^n )$ is equal to $\ell^m$. So we can find a common defining set $\Lambda_{n}$ with full $\mathbb P_x$-measure for all $x \in E_1$ such that for $\omega \in \Lambda_{n}$ and $t \ge 0$ it holds
 \begin{align*}
\ell^m_t(\omega) = (1_{K_m} \cdot \ell^n)_t \, (\omega) \quad \text{for} \ n > m.
\end{align*}
Set $\Lambda := \bigcap_{n \in \N} \Lambda_n$.
Then we define $(\ell_t)_{t \ge 0}$ by:
\begin{align*}
\ell_t(\omega)=
\begin{cases}
 \ell^{n}_t(\omega), & \tau_{n-1}(\omega) < t \le \tau_{n}(\omega), \quad n \in \N \\
 \ell_{\lifetime^{-}}(\omega), & t \ge \lifetime(\omega),
\end{cases}
\end{align*}
for $\omega \in \Lambda$ and $0$ else,
with $\tau_n$ the exit times of $K_n$, $n \in \N$, and $\tau_0 := 0$.
Following the proof of \cite[Lem.~5.1.8]{FOT11} we get that $(\ell_t)_{t \ge 0}$ is a strict PCAF being in Revuz correspondence to $\nu$.
The integral identities follow directly by definition of $(\ell_t)_{t \ge 0}$.
\flushright\end{proof}

%
%

\begin{remark}
Note that the only possibility for $l_t(\omega)$ being infinite is that $t \ge \lifetime(\omega)$ for $\omega \in \Lambda$.
\end{remark}

\begin{remark} \label{RemLocaltimeGrows}
We can conclude from the construction of $(\ell_t)_{t \ge 0}$ that the functional grows only when $\mathbf X_t$, $t \ge 0$, is at the boundary part $\Gamma_2 \cap \{ \varrho > 0 \}$.
Indeed, let $n \in \N$. Then by Theorem \ref{TheoLocaltimeS1} we get for $t \ge 0$
\begin{align*}
\int_0^t 1_{\Gamma_2 \cap \{ \varrho > 0 \}}  (\mathbf X_s) 1_{K_n} (\mathbf X_s) \, d \ell_s = \int_0^t 1_{\Gamma_2 \cap \{ \varrho > 0 \}}  (\mathbf X_s) d \ell^n_s = (1_{\Gamma_2 \cap \{ \varrho > 0 \}} \cdot \ell^n)_t. 
\end{align*}
By Lemma \ref{LemmadfAt} we get that $(1_{\Gamma_2 \cap \{ \varrho > 0 \}} \cdot \ell^n)$ is in Revuz correspondence to $1_{\Gamma_2 \cap \{ \varrho > 0 \}} 1_{K_n} \sigma = 1_{K_n} \sigma$. Thus $(1_{\Gamma_2 \cap \{ \varrho > 0 \}} \cdot \ell^n) = \ell^n$.

Altogether, we get for $t \ge 0$
\begin{align*}
\int_0^t 1_{\Gamma_2 \cap \{ \varrho > 0 \}}  (\mathbf X_s) 1_{K_n} (\mathbf X_s) \, d \ell_s = \ell^n_t.
\end{align*}
Letting $n$ tend to $\infty$ we obtain for $t \ge 0$
\begin{align*}
\int_0^t 1_{\Gamma_2 \cap \{ \varrho > 0 \}}  (\mathbf X_s) \, d \ell_s = \ell_t.
\end{align*}
Thus for $t \ge 0$
\begin{align*}
\int_0^t 1_{\Om \cup \, (\partial \Om \setminus \Gamma_2 \cap \{ \varrho > 0 \})} (\mathbf X_s) \, d \ell_s = \int_0^t 1_{\Omclo} (\mathbf X_s) \, d \ell_s -  \int_0^t 1_{ \Gamma_2 \cap \{ \varrho > 0 \}} (\mathbf X_s) \, d \ell_s = \int_0^t 1 \, d \ell_s -  \ell_t = 0. 
\end{align*}
\end{remark}

$ $\\[1.25ex]
To discuss the martingale solution property we also need to consider integration of functions on paths of the process with respect to the deterministic time. Here we can allow certain singularities. First note that $\mu$ is in Revuz correspondence to the additive functional $(\omega,t) \mapsto t$. Indeed, we have $U^\alpha_t f = R_\alpha f$ for every $f \in \mathcal B^+_b(E)$, $\alpha > 0$. Since $R_\alpha$ is symmetric, we get by \eqref{EqRevuzCorrespondence} the Revuz correspondence.

We introduce the notion of bounded variation.\index{bounded variation}\index{bounded variation!locally of}
\begin{definition} \label{DefBoundedVariation}
Let $g: \R^+ \to \R$ be a function. We say that $g$ is of \emph{bounded variation} up to time $T$, $0 \le T \le \infty$ if
\begin{align*}
\underset{n \in \N}{\sup} \left \{ \sum_{i=0}^{n-1} |g(t_{i+1}) - g(t_i)| \, \bigg | \, 0 = t_0 \le t_1 \le ... \le t_n = T \right \} < \infty.
\end{align*}

Let $(\mathbf \Om,\mathcal F,\mathbf Q)$ be a probability space. Let $\mathbf G:= (\mathbf G_t)_{t \ge 0}$ be an $\R$-valued stochastic process defined on $\mOm$. We say that $\mathbf G$ is \emph{locally of bounded variation} if for every $0 \le T < \infty$ it holds that the function $[0,\infty) \ni t \mapsto \mathbf G_t(\omega) \in \R$ is of bounded variation up to $T$ for $\mathbf Q$-a.e.~$\omega \in \mOm$.

Let $(\mathcal F_t)_{t \ge 0}$ be an filtration and $\tau$ be an $\mathcal F_t$-stopping time.
We say that $\mathbf G$ is \emph{locally of bounded variation up to} $\tau$ if there exists a sequence of $\mathcal F_t$-stopping times $(\tau_n)_{n \in \N}$ with $\tau_n \uparrow \tau$ such that $\mathbf G_{\cdot \minw \tau_n}$ is locally of bounded variation for $n \in \N$.
\end{definition}

\begin{theorem} \label{TheoremLpIntegraldt}
(i) Let $f \in L^p(E,\mu)$. Define $A := (f \cdot t) := (f \cdot t)_{t \ge 0}$ by
\begin{align}
A_t := (f \cdot t)_t := \int_0^t f(\mathbf X_s) ds. \label{EqDefNDt}
\end{align}

Then $f \cdot t$ is a strict finite CAF on $E_1$. Furthermore, $\mathbb E_x[ |A_t|] < \infty$ for every $0 \le t < \infty$, $x \in E_1$, and $f \cdot t$ is locally of bounded variation.

If $f$ is positive, $(f \cdot t)$ is in Revuz correspondence to $f \, \mu$.
\\
(ii) Let $f \in L^p(E,\mu)$, $\supp[f] \subset \subset U_n$ for one $n \in \N$, $U_n$ as in \eqref{EqDefUnLocaltime}.
Then
\begin{align*}
\mathbb E_x \left[(f \cdot t)^2_t \right] < \infty \quad \text{for} \ x \in E_1. 
\end{align*}
\\
(iii) If $f \in L^p_{\text{loc}}(E_1,\mu)$, then there exists a local strict CAF $A := (A_t)_{t \ge 0}$ that is $\mathbb P_x$-a.s.~equal to the integral in \eqref{EqDefNDt} for $t < \lifetime$. Moreover, $A$ is locally of bounded variation up to $\lifetime$. 
\end{theorem}
\begin{proof}
(i):  
Note that the $\mathcal L^p$-strong Feller property implies $R_1 g \, (x) < \infty$ for $x \in E_1$ and $g \in \mathcal L^p(E,\mu)$. So for $x \in E_1$
\begin{align*}
\mathbb E_x \left[\int_0^T |f|(\mathbf X_r) d r \right] \le e^T \mathbb E_x \left[ \int_0^\infty e^{-r} |f|(\mathbf X_r) dr \right] = e^T R_1 |f| \,(x) < \infty.
\end{align*}
So $\mathbb E_x[ \int_0^T |f| (\mathbf X_r) d r] < \infty$ and $\int_0^T |f|(\mathbf X_r) dr < \infty$ $\mathbb P_x$-a.s.~for $x \in E_1$ and $0 \le T < \infty$. 
This holds for every $0 \le T < \infty$. 
Define
\begin{align*}
\Lambda_f := \bigg \{ \omega \in \mathbf \Om^1 \, | \, \int_0^N |f|(\mathbf X_r(\omega)) dr < \infty \ \text{for all} \ N \in \N \bigg \}
\end{align*}
and for $t \ge 0$
\begin{align*}
(f \cdot t)_t \, := 
\begin{cases}
\int_0^t f(\mathbf X_s(\omega)) \, d s & \quad \text{if} \ \omega \in \Lambda_f \\
0 & \quad \text{else}.
\end{cases}
\end{align*}
Note that $\Lambda_f$ is shift-invariant and $\mathbb P_x(\Lambda_f)=1$ for $x \in E_1$. Then we get with a standard calculation that $f \cdot t$ defines a finite PCAF with defining set $\Lambda_f$.

For every partition $0 = t_0 \le t_1 \le ... \le t_n = t < \infty$, $n \in \N$ we have
\begin{align*}
\sum_{i=0}^{n-1} \left| \int_{t_{i}}^{t_{i+1}} f(\mathbf X_s) \, ds \, \right| \le \int_0^t |f(X_s)| \, ds < \infty \quad \mathbb P_x-\text{a.s. for} \, x \in E_1.
\end{align*}
So $f \cdot t$ is locally of bounded variation.

Assume that $f$ is positive. Let $g \in \mathcal B^+(E_1)$. We have 
\begin{multline*}
U^\alpha_{f \cdot t} g \, (x) = \mathbb E_x \left[\int_0^\infty \exp(-\alpha s) \, g(\mathbf X_s) \, d (f \cdot t)_s \right] \\
= \mathbb E_x \left[\int_0^\infty \exp(-\alpha s) \, (f g)(\mathbf X_s) \, ds \right] = R_\alpha f g \, (x) \quad \text{for} \ \alpha > 0 \ \text{and} \ x \in E_1.
\end{multline*}
Since $(R_\alpha)_{\alpha > 0}$ is symmetric, we have for all $h \in \mathcal B_b^+(E_1)$ and $\alpha > 0$:
\begin{align*}
\int_{E_1} h \, (U^{\alpha}_{f \cdot t} g) \, d \mu = \int_{E_1} h \, R_{\alpha} f  g \, d \mu =  \int_{E_1} R_{\alpha} h \, g f \, d \mu.
\end{align*}
So by \eqref{EqRevuzCorrespondence} we get that $(f \cdot t)$ is in Revuz correspondence to $f \mu$.
\\
(ii): With a similar calculation as in \cite[p.~245]{FOT11} we get for $0 \le t < \infty$ and $x \in E_1$
\begin{align*}
\mathbb E_x \left[ (f \cdot t)^2_t \right] \le 2 e^t \mathbb E_x \left[\int_0^t |f|(\mathbf X_s) \, R_1 |f|(\mathbf X_s) \, d s \right].
\end{align*}
Set $h(x) := |f| R_1|f| \, (x)$. Since the resolvent is $\mathcal L^p$-strong Feller on $E_1$ (see Theorem \ref{TheoDiffProcess}), we have that $R_1 |f|$ is continuous and hence locally bounded on $E_1$. In particular, $R_1 |f|$ is bounded on the compact support of $f$. Thus $h \in L^p(E,\mu)$. So as above $\mathbb E_x[\int_0^t h(\mathbf X_s) ds] < \infty$ for $0 \le t < \infty$ and $x \in E_1$.
\\
(iii): Now assume that $f \in L^p_{\text{loc}}(E_1,\mu)$. Set $f_n = 1_{K_n} f$, $n \in \N$, $K_n$ as in \eqref{EqDefUnLocaltime}. Define $A^{f_n}$ to be the corresponding additive functional from $\eqref{EqDefNDt}$. So $A^{f_n}$ is a continuous additive functional with defining set $\Lambda_{f_n}$.
Let $\Lambda := \bigcap_{n \in \N} \Lambda_{f_n} \cap \{ \omega \in \mathbf \Om^1 \, | \, \lim_{n \to \infty} \tau_{n} \ge \lifetime \}$, $\tau_n$ the exit time of $K_n$, $n \in \N$.
Define
\begin{align*}
A_t := 1_{\Lambda} 1_{\{t < \lifetime\}} \, \limn A^{f_n}_t, t \ge 0. 
\end{align*}
Let $t > 0$ and $\omega \in \Lambda \cap \{ t < \lifetime \}$. There exists $n_0 \in \N$ such that $t < \tau_{n_0} < \lifetime$. Then 
\begin{align*}
\limn A^{f_n}_t(\omega) = A^{f_{n_0}}_t(\omega).
\end{align*}
So $(A_t)_{t \ge 0}$ is well-defined and $\mathcal F^1_t$-adapted. 

Since $A^f(\omega)$ equals $A^{f_{n_0}}(\omega)$ on $[0,t]$ for $t < \tau_{n_0}$, it is therefore continuous and additive on $[0,t]$.
So $A^f$ is a local strict CAF. Furthermore, $(f \cdot t)_{\cdot \minw \tau_n} = (f_n \cdot t)_{\cdot \minw \tau_n}$, $n \in \N$. So $(f \cdot t)_{t}$ is locally of bounded variation up to $\lifetime$.
\flushright
\end{proof}

\begin{lemma} \label{LemmaLocaltimeIntegral}
Let $g$ be $\mathcal B(\Gamma_2)$-measurable and bounded, $\supp[g] \subset \subset U_n$, for one $n \in \N$.
Then $g \cdot \ell$, defined by
\begin{align*}
(g \cdot \ell)_t := \int_0^t g(\mathbf X_s) d \ell_s, \quad t \ge 0,
\end{align*}
is a strict finite CAF with defining set of $\ell$ and it holds
\begin{align}
\mathbb E_x \left[|(g \cdot \ell)_t| \right] < \infty \ \text{and} \ \mathbb E_x \left[(g \cdot \ell)^2_t \right] < \infty \quad \text{for} \ 0 \le t < \infty \ \text{and} \ x \in E_1. \label{EqSquareNells}
\end{align}
Furthermore, $g \cdot \ell$ is locally of bounded variation.
Assume that $g \in \mathcal B(\Gamma_2)$ is only locally bounded. Then $g \cdot \ell$ is a local strict CAF and locally of bounded variation up to $\lifetime$.
\end{lemma}
\begin{proof}
First extend $g$ to a function in $\mathcal B(E)$ in the trivial way, i.e., replace $g$ by $1_{\Gamma_2} g$. From Theorem \ref{TheoLocaltimeS1} we get $g \cdot \ell = g \cdot \ell^n$. So by Lemma \ref{LemmadfAt} we get that $g \cdot \ell$ is a strict finite PCAF on $E_1$ and \eqref{EqSquareNells} holds.
That $g \cdot \ell$ is locally of bounded variation, follows similarly as in the proof of Theorem \ref{TheoremLpIntegraldt}.
The statements for $g$ being only locally bounded follow now with the same localizing procedure as in the proof of Theorem \ref{TheoremLpIntegraldt}(iii).\flushright\end{proof}

Let us introduce two classes of functionals, according to \cite{FOT11} but refined to pointwise properties.
The Skorokhod decomposition is formulated in terms of these classes.\index{$\mathcal M_c$}
Define
\begin{multline*}
\hspace{30pt}\mathcal M_c := \bigg \{ M : \mathbf \Om^1 \times [0,\infty) \to \R \, \bigg | \, M = (M_t)_{t \ge 0} \ \text{is a strict finite CAF},\\
\, \mathbb E_x[M^2_t] < \infty, 
\mathbb E_x[M_t] = 0 \ \text{for every} \ t \ge 0 \ \text{and} \ x \in E_1 \bigg \}\hspace{30pt}
\end{multline*}
and\index{$\mathcal N_c$}
\begin{multline}
\hspace{30pt}
\mathcal N_c := \bigg \{ N : \mathbf \Om^1 \times [0,\infty) \to \R \, \bigg | \, N = (N_t)_{t \ge 0} \ \text{is a strict finite CAF}, \\ e(N)=0, 
\mathbb E_x[|N_t|] < \infty \ \text{for every} \ t \ge 0 \ \text{and} \ x \in E_1 \bigg \}\hspace{15pt} \label{DefNc}
\end{multline}
with\index{$e(N)$}
\begin{align*}
e(N) := \underset{t \downarrow 0}{\lim} \frac{1}{2t} \mathbb E_\mu[N^2_t].
\end{align*}
The term $e(N)$ is called the \emph{energy} of $N$.
Note that the properties required in $\mathcal M_c$ and $\mathcal N_c$ are pointwise properties except for the zero energy requirement.

If $M \in \mathcal M_c$, then additivity together with $\mathbb E_x[M_t]=0$ imply that $M$ is a martingale under $\mathbb P_x$ for every $x \in E_1$.
Recall the definition of the operator $\Ltil$ on $\DNeu$. We may extend this definition to all functions $u \in C^2_c(E_1)$ and define
\begin{align}
\Ltil u = \sum^d_{i,j=1} a_{ij} \pdi \pdj u + \sum^d_{j=1} \left(\sum^d_{i=1} \pdi a_{ij} + \sum^d_{i=1} \frac{1}{\varrho} a_{ji} \pdi \varrho \right) \pdj u. \label{EqWidetildeLC2c}
\end{align}

We obtain the following theorem using \cite[Theo.~5.2.4]{FOT11} and \cite[Theo.~5.2.5]{FOT11}.
\begin{theorem} \label{TheoMartingaleWithLocalTime}

Let $u \in C^2_c(E_1)$. Let $N^{[u]}:=(N_t^{[u]})_{t \ge 0}$ with
\begin{align*}
N^{[u]}_t := \int_0^t \Ltil u \, (\mathbf X_s) \, d s -\int_0^t (A \nabla u, \unitnormal) \varrho \ (\mathbf X_s) d \ell_s, \quad t \ge 0.
\end{align*}
Then $N^{[u]} \in \mathcal N_c$ and is locally of bounded variation. Define $M^{[u]}:=(M_t^{[u]})_{t \ge 0}$ with
\begin{align*}
M^{[u]}_t := u(\mathbf X_t) - u(\mathbf X_0) - N^{[u]}_t.
\end{align*}
Then $M^{[u]} \in \mathcal M_c$, in particular it is an square-integrable $\mathcal F^1_t$-martingale starting at zero.
\end{theorem}
The integrals are defined in the sense of Theorem \ref{TheoremLpIntegraldt} and Lemma \ref{LemmaLocaltimeIntegral}.
As defining set $\Lambda$ for $M^{[u]}$ we take the intersection of the defining sets of $\widetilde{L}u \cdot t$ and $\ell$. Due to Theorem \ref{TheoremLpIntegraldt} and Lemma \ref{LemmaLocaltimeIntegral} $M^{[u]}$ is additive on this set and $\mathbb P_x(\Lambda)=1$ for every $x \in E_1$.


\begin{proof}
Let $u \in C^2_c(E_1)$.
From Lemma \ref{LemmaLocaltimeIntegral} and Theorem \ref{TheoremLpIntegraldt} together with the calculations on \cite[p.~244]{FOT11} we have $N^{[u]} \in \mathcal N_c$ and $N^{[u]}$ is locally of bounded variation.

Choose $K_N$ such that $\supp[u] \subset \subset U_N \subset K_N$.
Let $v \in D(\mathcal E)$. Then $v \in H^{1,2}(U_N \cap \Om)$. Since the support of $u$ has positive distance to the non-smooth boundary part of $\partial \Om$, we can apply the divergence theorem to obtain
\begin{align*}
\mathcal E(u,v) = \int_{\Om} (A \nabla u,\nabla v) d \mu = - \int_{\Om} \Ltil u \, v \, d \mu + \int_{\partial \Om} Tr(v) (A \nabla u, \unitnormal) \varrho d \sigma.
\end{align*}

Set either $g := (A \nabla u, \unitnormal)^+ \varrho$ or $g := (A \nabla u, \unitnormal)^- \varrho$.

Note that $g \cdot \ell = g \cdot \ell^N$ is in Revuz correspondence to $g \sigma$. By Theorem \ref{TheoremLpIntegraldt} we have that $(\Ltil u \cdot t)_t = \int_0^t \Ltil u \, (\mathbf X_s) \, ds$, $t \ge 0$, is a strict finite CAF. Moreover, $(\Ltil u)^{+/-} \cdot t$ is in Revuz correspondence to $(\Ltil u)^{+/-} \mu$.

Let $h \in L^1(E,\mu) \cap \mathcal B^+_b(E)$, set $v := R_1 h \in D(\mathcal E) \cap C(E_1)$.
So the Revuz correspondence implies by \cite[Theo.~5.1.3(vi)]{FOT11}
\begin{multline*}
\hspace{10pt}\lim_{t \to 0} \frac{1}{t} \mathbb E_{v \mu} \left[ (\Ltil u \cdot t)_{t} - ((A \nabla u,\unitnormal) \varrho \cdot \ell_t)_{t} \right] 
= \int_{\Om} \Ltil u \, v \, d \mu 
- \int_{\partial \Om} Tr(v) (A \nabla u,\unitnormal) \varrho d \sigma\hspace{-5pt} =\hspace{-5pt} - \mathcal E(u,v).\hspace{10pt}
\end{multline*} 

Thus from \cite[Theo.~5.2.4]{FOT11} we obtain
\begin{align}
\mathbb E_x \left[ N^{[u]}_t \right ] = P_t u(x) - u(x) \quad \text{for} \ \mu\text{-a.e.} \ x \in \Omclo. \label{EqNtPtuMuae} 
\end{align} 
Using the absolute continuity of $(P_t)_{t > 0}$ on $E_1$ we get
\begin{align*}
P_s \mathbb E_\cdot \left[ N^{[u]}_t \right ] \, (x) = P_s (P_t u - u) \, (x) \quad \text{for every} \ x \in E_1. 
\end{align*} 
The right-hand side converges to $P_t u(x) - u(x)$, as $s \to 0$, for every $x \in E_1$. This follows since $P_t u - u$ is a continuous bounded function on $E_1$ and the paths of $\mathbf M^1$ are right-continuous at zero.
Observe that $\Ltil u \cdot t = (\Ltil u)^+ \cdot t - (\Ltil u)^- \cdot t$, the analogous property holds for $(A \nabla u,\unitnormal) \varrho \cdot \ell$. Applying \eqref{EqConvergenceS} in Lemma \ref{LemmadfAt} with $f=1$ and $A = (\Ltil u)^{+/-}$ or $(A \nabla u,\unitnormal)^{+/-} \varrho \cdot \ell$ we get convergence of the left-hand side.
So altogether, we get that \eqref{EqNtPtuMuae} holds for every $x \in E_1$.

Using the Markov property of $\mathbf M^1$ we get from this that $M^{[u]}$ is a martingale starting at zero.
Since $u$ is bounded and $\mathbb E_x[(\Ltil u \cdot t)^2_t]$, $\mathbb E_x[(g \cdot \ell)^2_t] < \infty$ for $0 \le t < \infty$ and $x \in E_1$, we have that $M^{[u]}$ is square-integrable for every $x \in E_1$.
\flushright\end{proof}

Theorem \ref{TheoMartingaleWithLocalTime} yields that $M_t^{[u]} := u(\mathbf X_t) - u(\mathbf X_0) - N^{[u]}_t$ is a continuous square-integrable martingale. Hence $M^{[u]} \in \mathcal M_c$.
Next we further analyze this martingale by considering the quadratic variation process. We introduce the notion of local martingales, see \cite[Ch.~I, Def.~5.15]{KS91}.\index{martingale!local continuous}

\begin{definition} \label{DefLocalMartingale}
Let $(\mathbf \Om,\mathcal F,\mathbf Q)$ be a probability space with filtration $(\mathcal F_t)_{t \ge 0}$. We say that a stochastic process $M = (M_t)_{t \ge 0}$ is a \emph{continuous local martingale} up to a $\mathcal F_t$-stopping time $\tau$ if there exists a sequence of $\mathcal F_t$-stopping times $(\tau_n)_{n \in \N}$ with $\tau_n \uparrow \tau$, $\tau_n < \tau$ such that for every $n \in \N$ the stopped process $M^{\tau_n} = (M_{t \minw \tau_n})_{t \ge 0}$ is a continuous $\mathcal F_t$-martingale.
We say that $(\tau_n)_{n \in \N}$ \emph{reduces} $M$.
\end{definition}

From \cite[Ch.~I, Theo.~5.13 and Prob.~5.17]{KS91} we get the following theorem.
\begin{theorem} \label{TheoQuadraticVariation}
Let $(M_t)_{t \ge 0}$ be a local martingale up to a stopping time $\tau$ starting at zero that is continuous in $[0,\tau)$. Then there exists an adapted process $\sharpbracket{M} = (\sharpbracket{M}_t)_{t \ge 0}$, unique up to time $\tau$, the \textit{quadratic variation process},\index{quadratic variation} with the following properties.
\begin{sequence}
\item[(i)] $\sharpbracket{M}_0 = 0$ and $\sharpbracket{M}$ is increasing.
\item[(ii)] $\sharpbracket{M}$ is continuous in $[0,\tau)$.
\item[(iii)] $(M^2_{t} - \sharpbracket{M}_t)_{t \ge 0}$ is a local martingale up to time $\tau$. 
\end{sequence}

If $(M_t)_{t \ge 0}$ is square-integrable and continuous in $[0,\infty)$, then $(M^2_{t} - \sharpbracket{M}_t)_{t \ge 0}$ is a martingale.\index{martingale!square-integrable}
\end{theorem}

So for $u \in C^2_c(E_1)$, we get for $M^{[u]}$ an associated quadratic variation process $\sharpbracket{M^{[u]}}$. Note that $\sharpbracket{M^{[u]}}$, obtained from Theorem \ref{TheoQuadraticVariation}, is constructed for each $x \in E_1$ separately since we consider the measurable space endowed with the different probability measures $\mathbb P_x$, $x \in E_1$.
Following \cite[Theo.~A.3.17]{FOT11}, however, we can construct from this a process $\sharpbracket{M^{[u]}}$ that is a strict additive functional with common defining set for all $x \in E_1$. 

Using \cite[Theo.~5.2.3]{FOT11} we get an explicit representation for $\sharpbracket{\Mu}$.
\begin{theorem} \label{TheoremMartingaleEnergyMeasure}
Let $u \in C^2_c(E_1)$.
Then $\sharpbracket{\Mu} = 2 (A \nabla u,\nabla u) \cdot t$, i.e.,
\begin{equation*}
\Muts{u}{t} - 2 \int_0^t (A\nabla u,\nabla u)(\mathbf X_s) ds \quad \text{is an} \ \mathcal F^1_t\text{-martingale under} \, \mathbb P_x \ \text{for every} \ x \in E_1.
\end{equation*}
\end{theorem}

\begin{proof}
From Theorem \ref{TheoMartingaleWithLocalTime} we get that $\Mu$ is a continuous square-integrable martingale and a strict finite CAF which is strict on $E_1$. The calculation in \cite[Theo.~5.2.3]{FOT11} yields that the Revuz measure of $\sharpbracket{\Mu}$ is $\mu_{\sharpbracket{u}}$.
The energy measure $\mu_{\sharpbracket{u}}$ of $u$ is given by $\mu_{\sharpbracket{u}} = 2 (A \nabla u,\nabla u) \, \mu$, see e.g.~\cite[p.~254]{FOT11}.

Thus by Theorem \ref{TheoremConstrTildeS} we find a set $\widetilde{\Lambda} \subset \mathbf \Om^1$ with $\mathbb P_x(\widetilde{\Lambda}) = 1 \, \text{for all} \ x \in E_1$ and

\begin{align*}
2 ((A\nabla u,\nabla u) \cdot t)(\omega) = \sharpbracket{\Mu}(\omega) \quad \text{for all} \ \omega \in \widetilde{\Lambda}. 
\end{align*} \flushright\end{proof}

So we get for $u \in \mathcal D_1$ the \emph{Skorokhod decomposition}
\begin{align}
u(\mathbf X_t) - u(\mathbf X_0) = \Nu_t + \Mu_t \ \text{for} \ t \ge 0 \label{EqSkoro}
\end{align}
with $N^{[u]}$, $M^{[u]}$ as in Theorem \ref{TheoMartingaleWithLocalTime} and  $\sharpbracket{\Mu}$ as in Theorem \ref{TheoremMartingaleEnergyMeasure}. In particular, $(u(\mathbf X_t))_{t \ge 0}$ is a semimartingale, see Definition \ref{DefSemimartingale} below.

In order to study the behavior of the process $(\mathbf X_t)_{t \ge 0}$ we need also information of the joint behavior of $(u_1(\mathbf X_t))_{t \ge 0}$ and $(u_2(\mathbf X_t))_{t \ge 0}$ for $u_1,u_2 \in C^2_c(E_1)$.

%

For two martingales $M^{(1)}$ and $M^{(2)}$ define the quadratic \emph{covariation process} by
\begin{align*}
\langle M^{(1)}, M^{(2)} \rangle := \frac{1}{2} \left \{ \langle M^{(1)} + M^{(2)} \rangle - \langle M^{(1)} \rangle - \langle M^{(2)} \rangle \right \}.
\end{align*}
Observe that $\sharpbracket{\alpha M^{(1)}} = \alpha^2 \sharpbracket{M^{(1)}}$ for $\alpha \in \R$. Thus $\langle M^{(1)}, M^{(1)} \rangle = \sharpbracket{M^{(1)}}$.

Using that for $u_1, u_2 \in C^2_c(E_1)$ and $\alpha, \beta \in \R$ it holds 
$\Mu[\alpha u_1 + \beta u_2] = \alpha \Mu[u_1] + \beta \Mu[u_2]$ we can calculate the covariation to be 
%
%
\begin{align*}
\langle \Mu[u_1], \Mu[u_2] \rangle = 2 ( (A \nabla u_1,\nabla u_2) \cdot t),
\end{align*}
i.e.,
\begin{align*}
\langle \Mu[u_1], \Mu[u_2] \rangle_t = 2 \int_0^t (A \nabla u_1, \nabla u_2)(X_s) \, ds \ \quad \text{for all} \ t \ge 0 \ \mathbb P_x-a.s., \ x \in E_1.
\end{align*}
%

\section{Semimartingale Structure and Weak Solutions} \label{SecSemiMartingale}
In this section we study the coordinates of the process $\mathbf X_t = (\mathbf X^{(1)}_t,...,\mathbf X^{(d)}_t)$, rather than functions of $(\mathbf X_t)_{t \ge 0}$. 
Throughout this section we keep the same setting as in the introduction. In particular, we assume Conditions \ref{Cond1Matrix}, \ref{CondContinuity}, \ref{CondDiffDensity} and \ref{Cond1DiffBoundary}.
We show that $(\mathbf X_t)_{t \ge 0}$ is a semimartingale up to the lifetime $\lifetime$. Under the assumption that the process is conservative, we prove that it yields a weak solution to an singular SDE with reflection.
So we recall the definition of semimartingales first.

\begin{definition} \label{DefSemimartingale}
Let $(\mathbf \Om,\mathcal F,\mathbf Q)$ be a probability space with filtration $(\mathcal F_t)_{t \ge 0}$. Let $(\mathbf X_t)_{t \ge 0}$ be a stochastic process. Let $\tau$ be a stopping time. We say that $(\mathbf X_t)_{t \ge 0}$ is a \emph{continuous semimartingale}\index{semimartingale} up to $\tau$ if there exists $\mathcal F_t$-adapted processes $(M_t)_{t \ge 0}$ and $(N_t)_{t \ge 0}$, continuous in $[0,\tau)$, such that
$(M_t)_{t \ge 0}$ is a local $\mathcal F_t$-martingale up to $\tau$ and $(N_t)_{t \ge 0}$ is locally of bounded variation up to $\tau$ and
\begin{align*}
\mathbf X_{t \minw \tau} = \mathbf X_0 + N_{t \minw \tau} + M_{t \minw \tau} \quad \ \text{for} \ t \ge 0 \ \mathbf Q-\text{a.s.}
\end{align*}

\end{definition}
In order to apply our previous results, we have to transfer the properties from $u(\mathbf X_t)$, $u \in C^2_c(E_1)$, to $\mathbf X^{(i)}_t$, $1 \le i \le d$. This is done using localization arguments.

For $1 \le i \le d$ define $N^{(i)}$ by
\begin{align}
N_t^{(i)} = \int_0^t b_i(\mathbf X_s) \, ds - \int_0^t (e_i, A \unitnormal) \varrho \, (\mathbf X_s) \, d \ell_s, \quad t \ge 0, \label{EqNt}
\end{align}
in the sense of Theorem \ref{TheoremLpIntegraldt}(iii) and Lemma \ref{LemmaLocaltimeIntegral} where
\begin{align}
b_i(x) := \sum^d_{j=1} \pdj a_{ij}(x) + \sum^d_{j=1} \frac{1}{\varrho} a_{ij} \pdj \varrho \, (x), \quad x \in E_1, 1 \le i \le d, \label{EqFirstOrderCoeff}
\end{align}
are the first-order coefficients of $\Ltil$ from \eqref{EqWidetildeLC2c}.
Define $M^{(i)}$, $1 \le i \le d$, by
\begin{align}
M^{(i)}_t := \mathbf X_t^{(i)} - \mathbf X^{(i)}_0 - N_t^{(i)}, \quad t \ge 0. \label{EqMt}
\end{align}
Recall that $b_i$ is locally $L^p(E_1,\mu)$-integrable and $(e_i, A \unitnormal) \varrho$ is locally bounded for $1 \le i \le d$. By $e_i$, $1 \le i \le d$, we denote the i-th unit vector.\index{$e_i$|see{unit vector}}\index{unit vector}

Denote by $\Lambda$ the defining set of the boundary local time from Theorem \ref{TheoLocaltimeS1}.
According to Theorem \ref{TheoremLpIntegraldt} and Lemma \ref{LemmaLocaltimeIntegral} all $N^{(i)}$ and $M^{(i)}$, $1 \le i \le d$, form local strict CAF and we can find a common defining set $\widetilde{\Lambda} \subset \Lambda$.

\begin{theorem} \label{TheoLocalMartingalePx}
The processes $N^{(i)}$ and $M^{(i)}$, $1 \le i \le d$, are local strict CAF.
The processes $N^{(i)}$ are locally of bounded variation up to $\lifetime$ $\mathbb P_x$-a.s.~for every $x \in E_1$.
The processes $M^{(i)}$ are continuous local $\mathcal F^1_t$-martingales up to $\lifetime$ under $\mathbb P_x$, $x \in E_1$, with reducing sequence $\ttau_n := \tau_n \minw n \minw \lifetime$ where $\tau_n$ is the exit time of $K_n$, $n \in \N$, defined after \eqref{EqDefUnLocaltime}. 
The quadratic variation and covariation processes (up to $\lifetime$) are given by
\begin{align}
\langle M^{(i)}, M^{(j)} \rangle_{\cdot \minw \lifetime} = 2 \, (a_{ij} \cdot t)_{\cdot \minw \lifetime}, \ 1 \le i,j \le d, \label{EqQuadraticVarinTheo}
\end{align}
and have the same reducing sequence.
In particular, $(\mathbf X^{(i)}_t)_{t \ge 0}$, $1 \le i \le d$, are semimartingales up to $\lifetime$ with
\begin{align*}
\mathbf X^{(i)}_{t \minw \lifetime} = \mathbf X^{(i)}_0 + N^{(i)}_{t \minw \lifetime} + M^{(i)}_{t \minw \lifetime}  
\end{align*}
for $0 \le t < \infty$  $\mathbb P_x$-a.s.~for $x \in E_1$ and $1 \le i \le d$.

There exists a set $\hat{\Lambda}$ with $\mathbb P_x(\hat{\Lambda})=1$ for $x \in E_1$ with the following properties: The set $\hat{\Lambda}$ is contained in the defining sets of $N^{(i)}$ and $M^{(i)}$ for $1 \le i \le d$. Moreover, \eqref{EqQuadraticVarinTheo} hold on $\hat{\Lambda}$ and $N^{(i)}_{t \minw \lifetime}$ is given by the defining integral of \eqref{EqNt} on $\hat{\Lambda}$. Furthermore, the paths $[0,\lifetime) \ni t \mapsto N_t^{(i)}(\omega)$, $1 \le i \le d$, are locally of bounded variation for $\omega \in \hat{\Lambda}$.
\end{theorem}

\begin{proof}
Recall that $b_i$ is locally $L^p(E_1,\mu)$-integrable and $(e_i, A \unitnormal) \varrho$ is locally bounded for $1 \le i \le d$. 
So according to Theorem \ref{TheoremLpIntegraldt} and Lemma \ref{LemmaLocaltimeIntegral} $N^{(i)}$, $1 \le i \le d$, define local strict CAF up to $\lifetime$ and we can find a common defining set $\Lambda$. From the construction of the parts of $N^{(i)}$ it follows that the paths of $N^{(i)}$ are locally of bounded variation up to $\lifetime$ on $\Lambda$, compare the proof of Theorem \ref{TheoremLpIntegraldt}. The definition of $M^{(i)}$, $1 \le i \le d$, yields that they are also additive on $\Lambda$.

Choose a sequence of cutoff functions $\cutoff_n$, $n \in \N$, with $\cutoff_n = 1$ on $K_{n+1}$ and $\supp[\cutoff_n] \subset \subset U_{n+2}$, $(K_n)_{n \in \N}$ and $(U_n)_{n \in \N}$ as in \eqref{EqDefUnLocaltime}.

Define $u^{(n)}_i$, $n \in \N$, $1 \le i \le d$, with $u^{(n)}_i(x) = \cutoff_n(x) x_i$.
Then $u^{(n)}_i(x) = x_i$ in a neighborhood of $U_n$ since $\cutoff_n = 1$ on $U_{n+1}$ for $n \in \N$.
So for $x \in U_n$ it holds $\pdj u^{(n)}_i(x) = \delta_{ij}$ and $\pdj \pdk u^{(n)}_i(x) = 0$ for $1 \le i,j,k \le d$.
Thus $\Ltil u^{(n)}_i (x) = b_i(x)$ for $x \in U_n$ and $n \in \N$.
So we have for all $n \in \N$, $t \ge 0$ and $1 \le i \le d$
\begin{align}
\Nu[u^{(n)}_i]_{t \minw \tau_n} = \int_0^{t \minw \tau_n} b_i(\mathbf X_s) \, ds - \int_0^{t \minw \tau_n} \varrho (e_i, A \unitnormal)(\mathbf X_s) \, d \ell_s = N^{(i)}_{t \minw \tau_n}. \label{EqNStopped}
\end{align}
Using that $\ttau_n \le \tau_n$ we get from construction of $u^{(i)}_n$ together with \eqref{EqNStopped} for all $0 \le t < \infty$ and $n \in \N$
\begin{align*}
M^{(i)}_{t \minw \ttau_n} = \mathbf X^{(i)}_{t \minw \ttau_n} - \mathbf X^{(i)}_0 - \int_0^{t \minw \ttau_n} b_i (\mathbf X_s) \, ds + \int_0^{t \minw \ttau_n} \varrho (e_i, A \unitnormal)(\mathbf X_s) \, d \ell_s = \Mu[u^{(n)}_i]_{t \minw \ttau_n}
\end{align*}
for $1 \le i \le d$.
So $(M^{(i)}_{t \minw \ttau_n})_{t \ge 0}$ is a martingale for every $n \in \N$ and $1 \le i \le d$. Thus $M^{(i)}$ is a continuous local martingale with reducing sequence $(\ttau_n)_{n \in \N}$ for $1 \le i \le d$. 

Define $A^{(i)}_t := 2 \int_0^t a_{ii} (\mathbf X_s) ds$, $1 \le i \le d$, according to Lemma \ref{TheoremLpIntegraldt}. Then $A^{(i)}$, $1 \le i \le d$, is a local strict CAF up to $\lifetime$. We have by definition of the corresponding objects and Theorem \ref{TheoremMartingaleEnergyMeasure}
\begin{align}
(M^i_{t \minw \ttau_n})^2 - A_{t \minw \ttau_n}^{(i)} = (M^{[u^{(n)}_i]})^2_{t \minw \ttau_n} - \sharpbracket{M^{[u^{(n)}_i]}}_{t \minw \ttau_n} \label{EqMitStopped}
\end{align}
for $t \ge 0$, $1 \le i \le d$ and $n \in \N$.
So $(M^i)^2 - A^{(i)}$ is a continuous local martingale with reducing sequence $(\ttau_n)_{n \in \N}$. Thus the quadratic variation of $(M^{(i)}_t)_{t \ge 0}$ is given by $A^{(i)}$ for $1 \le i \le d$.

Now let $1 \le i,j \le d$, $i \neq j$. Then $M^i + M^j$ is a continuous local martingale as well. With the same argument we get that $A^{(i,j)}$ defined by $A^{(i,j)}_t := 2 \int_0^t a_{ii}(\mathbf X_s) + a_{jj}(\mathbf X_s) + 2 a_{ij}(\mathbf X_s) \, ds$, $t \ge 0$, is the corresponding quadratic variation process. 
%
%
Choose $\hat{\Lambda} \subset \Lambda$ such that for all $\omega \in \hat{\Lambda}$ it holds $\sharpbracket{M^{i}} = A^{(i)}$ and $\sharpbracket{M^{i} + M^{j}} = \widetilde{A}^{(i,j)}$. 
Altogether, we get for $0 \le t \le \lifetime$
\begin{equation*}
\sharpbracket{ M^{(i)}, M^{(j)}}_t = \frac{1}{2} \sharpbracket{M^i + M^j,M^i + M^j}_t - \frac{1}{2} \sharpbracket{M^i}_t - \frac{1}{2} \sharpbracket{M^j}_t 
= 2 \int_0^t a_{ij}(\mathbf X_s) \, ds.
\end{equation*}\flushright\end{proof}

Now we prove existence of weak solutions. 

As before we consider the process $\mathbf M^1$ obtained as the restriction of the $\mathcal L^p$-strong Feller process $\mathbf M$ from Theorem \ref{TheoDiffProcess} to $E_1 \cup \{ \Delta \}$. Note that if $\mathbf M$ is conservative then also $\mathbf M^1$\index{conservative!$\mathbf M^1$} is conservative. Conservativity of $\mathbf M$ holds e.g. if the coefficients fulfill certain growth conditions, see \cite[Theo.~5.7.3]{FOT11}.
Observe that due to \cite[Rem.~2.5]{BG13} we really have conservativity under $\mathbb P_x$ for every starting point $x \in E_1$.
So let us now assume that $\mathbf M^1$ is conservative, then we can prove existence of weak solutions, i.e., we prove Theorem \ref{TheoExWeakSolution} from the introduction. 

\begin{proof}[proof of Theorem \ref{TheoExWeakSolution}]\label{proofTheoExWeakSolution}
Define a probability measure $\mathbb P_{\mu_0}$ on $(\mOm^{1},\mathcal F^1)$ by
\begin{align*}
\mathbb{P}_{\mu_0} \, (\cdot) := \int_{E_1} \mathbb{P}_x \, (\cdot) \, d \mu_0(x).
\end{align*}
Obviously, $\mathcal L(\mathbf X_0)= \mu_0$ under $\mathbb P_{\mu_0}$.
By construction of $\mathbf M^1$ we have that $(\mathbf X_t)_{t \ge 0}$ has continuous paths on $[0,\infty)$. Furthermore, all paths stay in $E_1 \cup \{\Delta\}$. Since $\mathbb P_{\mu_0}(\lifetime = \infty)=1$, they do not hit $\Delta$. 
Let $N^{(i)}$ and $M^{(i)}$, $1 \le i \le d$, as defined before Theorem \ref{TheoLocalMartingalePx}.~Set~$M := (M^{(1)},...,M^{(d)})$.

Let $\hat{\Lambda} \subset \mOm^{1}$ as in Theorem \ref{TheoLocalMartingalePx}. Then the definition of $N^{(i)}$, $1 \le i \le d$, implies that on $\omega \in \hat{\Lambda}$ the integrals $\int_0^t |b_i(\mathbf X_s)| ds$, $1 \le i \le d$, and $\int_0^t |g_i(\mathbf X_s)| d \ell_s$, $1 \le i \le d$, exist and 
\begin{align*}
\mathbf X_t = \mathbf X_0 + \int_0^t \left(\nabla A + A \frac{\nabla \varrho}{\varrho} \right) \,  (\mathbf X_s) \, ds - \int_0^t \varrho A \unitnormal \, \, (\mathbf X_s) d \ell_s + M_t
\end{align*}
for all $0 \le t < \infty$ since $\lifetime = \infty$ by assumption.
 
Since $\mathbb P^1_x(\hat{\Lambda}) = 1$ for every $x \in E_1$, we have $\mathbb P_{\mu_0}(\hat{\Lambda})=1$. 
Thus this equality holds $\mathbb P_{\mu_0}$-a.s.

If $(C_t)_{t \ge 0}$ is an $\mathcal F^1_t$-martingale under $\mathbb P_x$ for every $x \in E_1$, then it is also one under $\mathbb P_{\mu_0}$.
For the reducing sequence $(\ttau_n)_{n \in \N}$ of $M^{(i)}$, $1 \le i \le d$, as in Theorem \ref{TheoLocalMartingalePx} we have $\ttau_n \uparrow \infty$ $\mathbb P_x$-a.s.~for every $x \in E_1$ hence also $\mathbb P_{\mu_0}$-a.s.
So $M^{(i)}$ is again a local martingale with the same quadratic variation as in Theorem \ref{TheoLocalMartingalePx} for $1 \le i \le d$. So it is left to construct a Brownian motion $(W_t)_{t \ge 0}$ such that $M_t = \int_0^t \sqrt{2} \sigma \, (\mathbf X_s) \, d W_s$.

Note that $M^{(i)}$, $1 \le i \le d$, are continuous local martingales with 
\begin{align*}
\sharpbracket{ M^{(i)}, M^{(j)}}_t = 2 \int_0^t a_{ij}(\mathbf X_s) \, ds = \int_0^t \sqrt{2} \sigma (\sqrt{2} \sigma)^\top \, (\mathbf X_s) \, ds \ \text{for} \ t \ge 0.
\end{align*}
So we can adapt the proof of \cite[Ch.~5, Prop.~4.6]{KS91}. Starting from (4.12) therein we conclude the existence of an $r$-dimensional Brownian motion (possibly on an extension of the probability space of $\mathbf M^1$) such that\\
$\int_0^t \sigma^2_{ij}(\mathbf X_s) \, d s < \infty$ $\mathbb P_{\mu_0}$-a.s.~for $t \ge 0$, $1 \le i \le d$ and $1 \le j \le r$ and
\begin{align*}
M_t = \int_0^t \sqrt{2} \sigma \, (\mathbf X_s) \, d W_s \ \text{for} \ t \ge 0.
\end{align*}\flushright\end{proof}

\section{Stochastic dynamics for particle systems with hydrodynamic interaction} \label{SecIPS}

In this section we consider systems of $N$ particles, $N \in \N$, which interact both through hydrodynamic interaction and direct interaction via pair-potentials. We adapt the model of \cite{TP86} to describe the interaction of colloidal particles suspended in a liquid.
Let $\Om_0 \subset \R^d$, $d \in \N$, $\partial \Om_0$ locally Lipschitz smooth and $\Gamma_2 \subset \partial \Om_0$ open in $\partial \Om_0$.
Assume that $\Gamma_2$ is $C^2$-smooth and $\Gamma_2 \subset \partial \Om_0$ has zero capacity w.r.t.~to the canonical gradient Dirichlet form on $\Om_0$, i.e., the closure of $\eqref{EqGradientDirForm}$ with $A=\Id$ and $\varrho=1$.
Define $\Lambda := \Omclo_0$.

Let $\pairpot :\R^d \to \R \cup \{ \infty \}$ be a symmetric pair potential, i.e.,~$\pairpot(-x) = \pairpot(x)$ which fulfills the following conditions. 

\begin{condition}\label{CondPotential}
\begin{sequence}
\item[(i)] For $dx$-a.e.~$x \in \R^d$ it holds $|\pairpot(x)|<\infty$ and for $x \to 0$ it holds $|\pairpot(x)| \to \infty$. 
\item[(ii)] The mapping $\R^d \to \R^+_0$, $x \mapsto \exp(-\pairpot(x)) =: \varrho_0(x)$ is continuous.
\item[(iii)] The function $\varrho_0$ is weakly differentiable on $\R^d$, $\exp(-\frac{\pairpot}{2}) \in H_{\text{loc}}^{1,2}(\R^d)$. $\pairpot$ is weakly \\ differentiable on $\R^d \setminus \{0\}$ and there exists $p > \frac{N d}{2}$ such that 
\begin{align}
 \nabla \pairpot \in L^p_{\text{loc}}(\{ \varrho_0 > 0 \},\exp(-\pairpot)dx). \label{CondPotentialInt}
\end{align}
\end{sequence}
\end{condition}
These are the same assumptions as in \cite{BG13}.

Let $A : \Lambda^N \to \R^{N d \times N d}$ be a continuously differentiable matrix-valued mapping of symmetric strictly elliptic matrices.
It is convenient to write $A$ as block-matrices $A = (\Aij{i}{j})_{1 \le i,j \le N}$ with $\Aij{i}{j} : {\Lambda}^N \to \R^{d \times d}$, $1 \le i,j \le N$.
We write an element $x \in \Lambda^N$ componentwise as $x=(x^{(1)},...,x^{(N)})$ with $x^{(i)} \in \Lambda$, $1 \le i \le N$. 

We describe the dynamics of the $N$ particles by a stochastic process $(\mathbf X_t)_{t \ge 0}$ written as $\mathbf X_t = (\mathbf X_t^{(1)},...,\mathbf X^{(N)}_t)$ with $\mathbf X^{(i)}_t \in \Omclo_0$ describing the position of the $i$-th particle.
The process should solve the following SDE interpreted in the \Ito \, sense:
\begin{align}
d \mathbf X^{(k)}_t = - \beta \sum_{j=1}^N \Aij{k}{j} \, (\mathbf X_t) \big( \sum_{\substack{l=1 \\ l \neq j}}^N (\nabla \pairpot) \big(\mathbf X^{(j)}_t - \mathbf X^{(l)}_t \big) \, dt + \unitnormal_{\Gamma_2} (\mathbf X^{(j)}_t) \, d \hat{\ell}^{(j)}_t \big) \nonumber \\
+ \sum_{j=1}^N \nabla_j \Aij{k}{j} \, (\mathbf X_t) \, dt + \sum_{j=1}^N \sqrt{2} \sigma^{(k,j)} \, (\mathbf X_t) \, d W^j_t, 1 \le k \le N, \label{FormalSDE}
\end{align}
where $\unitnormal_{\Gamma_2}$ denotes the outward unit normal at $\partial \Om_0$, $\hat{\ell}^{(j)}$, $1 \le j \le N$, denotes a later to be specified functional, that grows only, when the $j$-th particle is at the boundary. By $W^{(j)}$, $1 \le j \le N$, we denote a family of independent $\R^d$-valued Brownian motions and $\sigma = (\sigma^{(k,j)})_{1 \le k,j \le N}$ denotes a family of matrix-valued mappings with $\sigma^{(k,j)} : \Lambda^N \to \R^{d \times d}$, $1 \le k,j \le N$, such that $A = \sigma \sigma^\top$.
We define $\nabla_j A^{(k,j)}$ by $(\nabla_j A^{(k,j)})_{l} = \sum_{i=1}^d \partial_{x^{(j)}_i} A^{(k,j)}_{l,i}$, $1 \le l \le d$.

Here $\beta > 0$ denotes a constant, e.g. $\beta = \frac{1}{k_B T}$ with $k_B$ being Boltzmann's constant and $T$ the absolute temperature. 
\begin{remark}
We take the SDE from \cite[(2.23)]{TP86} and add an additional boundary term, that describes a repelling force from a wall-potential at the boundary of the state space. For a Fokker-Planck description of interacting particles with wall-potential, see e.g. \cite[(2.5)]{FJ82}.
The SDE describes the random evolution of the positions of $N$ colloidal particles which are suspended in a liquid. The matrix $A$ denotes the generalized diffusion matrix.

Note that the velocity of the particles is only implicitly treated via the so-called coarse-grained drift velocity
\begin{align*}
\overline{v}^{(k)} = \beta \sum_{j=1}^N \Aij{k}{j} F_j
\end{align*}
where $F_j$ denotes the force acting on the $j$-th particle, see \cite[(2.6)]{TP86}.
The force consists both of the direct interaction with the other particles and the repelling force caused by a wall-potential at the boundary of the state space.

Through the hydrodynamic interaction, i.e., interaction mediated through the surrounding liquid, the noise driving the several particles can be correlated.


For a further discussion of the equation, the related Smoluchowski equation and their physical background, see \cite{TP86}. 
We emphasize that due to the discussion on \cite[p.~604]{TP86} this SDE has indeed to be interpreted in the \Ito \, sense in order to be related to the corresponding Smoluchowski equation from \cite[(2.4)]{TP86}.

The specific shape of the (generalized) diffusion matrix depends of course on the concrete application. If hydrodynamic interaction is absent, the matrix is up to a constant just the identity matrix. For an example with hydrodynamic interaction, see e.g. \cite[(17)]{Zw69}.
\end{remark}

To solve this SDE we apply the results of our paper. So we first need to define a suitable Dirichlet form. Define $\varrho : \R^{N d} \to \R^+_0$ by
\begin{align*}
x \mapsto \frac{1}{Z} \exp \big(- \beta \sum_{1 \le i < j \le N} \pairpot (x^{(i)} - x^{(j)}) \big), 
\end{align*}
with $Z > 0$ a constant, e.g., the partition function. For this choice of $Z$, $\varrho$ denotes the canonical ensemble distribution, see \cite[(3.4)]{TP86}.
Set $\mu := \varrho \, dx$, the measure on $\Lambda^N$ with density $\varrho$ with respect to the Lebesgue measure. Define
\begin{align}
\mathcal E(u,v) := \int_{\Lambda^N} (A \nabla u, \nabla v) \, d \mu, \label{EqDirFormIPS}
\end{align}
\begin{align*}
\mathcal D := \{ u \in C_c(\Lambda^N) \, | \, u \in H^{1,1}_{\text{loc}}(\Om^N_0), \, \mathcal E(u,u) < \infty \}.
\end{align*}
Denote by $(\mathcal E,D(\mathcal E))$ the corresponding closure in $L^2(\Lambda^N, \mu)$. As in \cite{BG13} we identify a suitable state space for the $N$-particles process.

Set $\Lambda_s = \Om_0 \cup \Gamma_2$, we define the set of all \textit{admissible configurations} $\Lambda^N_{ad}$ by
\begin{multline*}
\hspace{17pt}\Lambda^N_{ad} := \{ \varrho > 0 \} \cap \{ (\xib{1},...,\xib{N}) \in \Lambda^N_s \ | \ \xib{k} \neq \xib{l} \ \text{for} \  k \neq l, 1 \le k,l \le N, \\
\text{there exists at most one} \ i \ \text{such that} \, \xib{i} \in \Gamma_2 \} \hspace{17pt}
\end{multline*}

So in an admissible configuration all particles are in $\Om_0$ or at the smooth boundary part $\Gamma_2$. Moreover, there are never two or more particles at the same place. Additionally, we exclude the case that two or more particles are at the boundary. This exclusion has to be done for technical reason, since the boundary of the configuration space is in general not smooth if two particles are located at the boundary. From \cite[Lem.~3.3]{BG13} we get that the boundary part $\Lambda^N_{\text{ad}} \cap \partial (\Lambda^N)$ is $C^2$-smooth. Furthermore, we can apply similar arguments as therein to conclude that $\Lambda^N_{\text{ad}}$ is complemented by set of zero capacity, see the proof of \cite[Lem.~3.3]{BG13} and \cite[Appendix A]{BG13}.

The Dirichlet form and the corresponding coefficients fulfill the Conditions \ref{Cond1Matrix}, \ref{CondContinuity}, \ref{CondDiffDensity} and \ref{Cond1DiffBoundary}. Let us assume from now on that the Dirichlet form is conservative.

So we may apply our previous results and obtain an $\mathcal L^p$-strong Feller diffusion process $(\mathbf X_t)_{t \ge 0}$ with values in $\Lambda^N_{\text{ad}} \cup \{ \Delta \}$. From Theorem \ref{TheoLocalMartingalePx} we get a Skorokhod decomposition. We write the corresponding functional $\mN$ as $\mN = (\mN^{(1)},...,\mN^{(N)})$ with $\mN^{(k)} : \mathbf \Om^1 \times [0,\infty) \to \R^d$, $1 \le k \le N$. Similarly we write $M = (M^{(1)},...,M^{(N)})$.
Let us first identify $\mN$. We may write the coefficients $b_i$, $1 \le i \le N d$ from \eqref{EqNt} in the following form. Let $b = (b^{(1)},...,b^{(N)})$, $b^{(k)} = (b^{(k)}_1,...,b^{(k)}_d)$ the drift term corresponding to the $k$-th particle, $1 \le k \le N$. Then
\begin{align*}
b^{(k)}_i (x) = \sum_{j=1}^{N} \sum_{l=1}^{d} \partial_{x^{(j)}_l} \Aij{k}{j}_{i,l}(x) + \sum_{j=1}^{N} \sum_{l=1}^{d} \big( \frac{1}{\varrho} \Aij{k}{j}_{i,l} \partial_{x^{(j)}_l} \varrho \big) (x).
\end{align*}
So
\begin{align*}
b^{(k)}(x) = \sum_{j=1}^{N} \nabla_j \Aij{k}{j} \, (x) + \sum_{j=1}^N \big( \Aij{k}{j} \frac{\nabla_j \varrho}{\varrho} \big) \, (x)
\end{align*}
with $\nabla_j \Aij{k}{j}$ defined after \eqref{FormalSDE} and $\nabla_j \varrho := (\partial_{x^{(j)}_1} \varrho,...,\partial_{x^{(j)}_d} \varrho)$.
The definition of $\varrho$ yields
\begin{align*}
\frac{\nabla_j \varrho}{\varrho} \, (x) = - \beta \sum_{\substack{l=1 \\ l \neq j}}^N (\nabla \pairpot) ( x^{(j)} - x^{(l)}).
\end{align*}
Thus
\begin{align*}
b^{(k)}(x) = \sum_{j=1}^{N} \nabla_j \Aij{k}{j} \, (x) - \beta \sum_{j=1}^N \Aij{k}{j} \, (x) \sum_{\substack{l=1 \\ l \neq j}}^N (\nabla \pairpot) ( x^{(j)} - x^{(l)}).
\end{align*}
Altogether we get for $1 \le k \le N$
\begin{align*}
\mN^{(k)}_t = \int_0^t \sum_{j=1}^{N} \nabla_j \Aij{k}{j} \, (\mathbf X_s) - \beta \sum_{j=1}^N \Aij{k}{j} \, (\mathbf X_s) \sum_{\substack{l=1 \\ l \neq j}}^N (\nabla \pairpot) ( \mathbf X^{(j)}_s - \mathbf X^{(l)}_s)  d s \\
- \int_0^t \sum_{j=1}^N (\Aij{k}{j} \unitnormal^{(j)} \varrho) \ (\mathbf X_s)  d \ell_s.
\end{align*}
Here $\unitnormal = (\unitnormal^{(1)},...,\unitnormal^{(N)})$ denotes the outward unit normal at $\Lambda^N_{\text{ad}} \cap \partial (\Lambda^N)$.
We rewrite the local time $\ell$ of the process $(\mathbf X_t)_{t \ge 0}$ into local times corresponding to the visits of the several particles. 

Note that
\begin{align*}
\Lambda^N_{\text{ad}} \cap \partial(\Lambda^N) \cap \{ \varrho > 0 \} = \{ \varrho > 0 \} \cap \bigcup_{k=1}^N \Om^{k-1}_0 \times \Gamma_2 \times \Om^{N-k}_0.
\end{align*}
Define $T^{(k)} := \Om^{k-1}_0 \times \Gamma_2 \times \Om^{N-k}_0$, $1 \le k \le N$. For $x = (x^{(1)},...,x^{(N)}) \in T^{(k)}$ we have for the outward unit normal $\unitnormal = (0,...,0,\unitnormal_{\Gamma_2}(x^{(k)}),0,...,0)$, i.e., in the $k$-th coordinate we have the outward unit normal of the boundary of the state space of the $k$-th particle.
Set $\hat{\ell}^{(k)} := \frac{1}{\beta} \varrho 1_{T^{(k)}} \cdot \ell$, $1 \le k \le N$. Then $(\hat{\ell}^{(k)})_{t \ge 0}$ grows only when the $k$-th particle is at the boundary. Furthermore, $\varrho \cdot \ell = \beta \sum_{k=1}^N \hat{\ell}^{(k)}$.
Thus for $1 \le k \le N$,
\begin{align*}
\int_0^t \sum_{j=1}^N \big(\Aij{k}{j} \unitnormal^{(j)} \varrho \big) \, (\mathbf X_s)  d \ell_s = \beta \sum_{j=1}^N \int_0^t \big(\Aij{k}{j} \unitnormal_{\Gamma_2} \big)(\mathbf X^{(j)}_s) \, d \hat{\ell}^{(j)}.
\end{align*}

For the martingale part $M$ we get for $1 \le i,j \le d$ and $1 \le k,l \le N$
\begin{align*}
\langle M^{(k)}_i, M^{(l)}_j \rangle_{t \minw \lifetime} = 2 \, \int_0^{t \minw \lifetime} \Aij{k}{l}_{i,j} \, (\mathbf X_s) \, ds. 
\end{align*}

Assuming that $(\mathbf X_t)_{t \ge 0}$ is conservative we can apply Theorem \ref{TheoExWeakSolution} to conclude existence of a weak solution. So for a given initial distribution $\mu_0 \in \mathcal P(\Lambda^N_{\text{ad}})$ we have that $(\mathbf X_t)_{t \ge 0}$ fulfills $\mathbb P_{\mu_0}$ almost surely with $1 \le k \le N$:
\begin{multline*}
\hspace{15pt}\mathbf X^{(k)}_t = \mathbf X^{(k)}_0 
+ \int_0^t \sum_{j=1}^{N} \nabla_j \Aij{k}{j} \, (\mathbf X_s) 
- \beta \sum_{j=1}^N \Aij{k}{j} \, (\mathbf X_s) \sum_{\substack{l=1 \\ l \neq j}}^N (\nabla \pairpot) ( \mathbf X^{(j)}_s - \mathbf X^{(l)}_s)  d s \\
- \beta \sum_{j=1}^N \int_0^t \Aij{k}{j} \unitnormal_{\Gamma_2}(\mathbf X^{(j)}_s) \, d \hat{\ell}^{(j)}_s 
+ \int_0^t \sqrt{2} \sum_{j=1}^N \sigma^{(k,j)} (\mathbf X_s) \, d W^{(j)}_s.\hspace{17pt} 
\end{multline*}

Furthermore, the process stays in the state space $\Lambda^N_{ad}$. So summarizing we have constructed a stochastic process describing the dynamics of interacting particles with hydrodynamic and direct interaction.

\section*{Acknowledgement(s)}
We thank Benedikt Heinrich for helpful and enriching discussions. Furthermore, we thank two unknown referees for valueable suggestions on the paper.


\begin{thebibliography}{12}



\bibitem{AD75}
R.~A. Adams.
\newblock {\em {Sobolev spaces.}}
\newblock {Pure and Applied Mathematics, 65. A Series of Monographs and
  Textbooks. New York-San Francisco-London: Academic Press, Inc., a subsidiary
  of Harcourt Brace Jovanovich}, 1975.

\bibitem{AKR03}
{S.~Albeverio, Y.~Kondratiev and M.~R{\"o}ckner}.
\newblock {\em {Strong Feller properties for distorted Brownian motion and
  applications to finite particle systems with singular interactions.}}
\newblock {Finite and infinite dimensional
  analysis in honor of Leonard Gross, volume 317 of \textit{Contemporary Mathematics}. 
  Amer. Math. Soc., Providence, RI, 2003.}

\bibitem{BB08}
R.~F. Bass and K.~Burdzy.
\newblock {On pathwise uniqueness for reflecting Brownian motion in
  $C^{1+\gamma}$ domains.}
\newblock {\em {Ann. Probab.}}, 36(6):2311--2331, 2008.

\bibitem{BaHs00}
R.~F. Bass and E.~P. Hsu.
\newblock {Pathwise uniqueness for reflecting Brownian motion in Euclidean
  domains.}
\newblock {\em Probab. Theory Relat. Fields}, 117(2):183--200, 2000.

\bibitem{BaHs90}
R.~F. Bass and P.~Hsu.
\newblock {The semimartingale structure of reflecting Brownian motion.}
\newblock {\em Proc. Am. Math. Soc.}, 108(4):1007--1010, 1990.

\bibitem{BaHs91}
R.~F. Bass and P.~Hsu.
\newblock {Some potential theory for reflecting Brownian motion in H\"older and
  Lipschitz domains.}
\newblock {\em Ann. Probab.}, 19(2):486--508, 1991.

\bibitem{Ba14}
B.~Baur.
\newblock {Elliptic boundary value problems and construction of $L^p$-strong Feller processes with singular drift and reflection. }
\newblock {\em Wiesbaden: Springer Spektrum; Kaiserslautern: TU Kaiserslautern (Diss. 2013)}, 2014. 

\bibitem{BG13}
B.~Baur and M.~Grothaus.
\newblock Construction and strong feller property of distorted elliptic
  diffusion with reflecting boundary.
\newblock {\em Potential Analysis}, 40(4):391--425, 2014.

\bibitem{BGS13}
B.~Baur, M.~Grothaus, and P.~Stilgenbauer.
\newblock {Construction of $\mathcal L^p$-strong Feller Processes via Dirichlet
  Forms and Applications to Elliptic Diffusions}.
\newblock {\em Potential Analysis}, 38(4):1233--1258, 2013.

\bibitem{BG68}
R.~M. Blumenthal and R.K. Getoor.
\newblock {\em {Markov processes and potential theory.}}
\newblock {Pure and Applied Mathematics, 29. A Series of Monographs and
  Textbooks. New York-London: Academic Press.}, 1968.

\bibitem{BBC05}
R.~F. Bass,~K.~Burdzy and Z.-Q. Chen.
\newblock {Uniqueness for reflecting Brownian motion in Lip domains.}
\newblock {\em {Ann. Inst. Henri Poincar\'e, Probab. Stat.}}, 41(2):197--235,
  2005.


\bibitem{Che93}
Z.-Q. Chen.
\newblock {On reflecting diffusion processes and Skorokhod decompositions.}
\newblock {\em Probab. Theory Relat. Fields}, 94(3):281--315, 1993.

\bibitem{DI93}
P. Dupuis and H. Ishii.
\newblock {SDEs with oblique reflection on nonsmooth domains.}
\newblock {\em {Ann. Probab.}}, 21(1):554--580, 1993.

\bibitem{FG07}
T.~Fattler and M.~Grothaus.
\newblock {Strong Feller properties for distorted Brownian motion with
  reflecting boundary condition and an application to continuous $N$-particle
  systems with singular interactions.}
\newblock {\em J. Funct. Anal.}, 246(2):217--241, 2007.


\bibitem{FJ82}
B.~U. Felderhof and R.~B. Jones.
\newblock {Linear response theory of sedimentation and diffusion in a
  suspension of spherical particles.}
\newblock {\em Physica A}, 119:591--608, 1983.

\bibitem{FOT94}
M.~Fukushima, Y.~Oshima, and M.~Takeda.
\newblock {\em {Dirichlet forms and symmetric Markov processes. 2nd revised and
  extended ed.}}
\newblock {de Gruyter Studies in Mathematics. Berlin: Walter de Gruyter}, 1994.

\bibitem{FOT11}
M.~Fukushima, Y.~Oshima, and M.~Takeda.
\newblock {\em {Dirichlet forms and symmetric Markov processes. 2nd revised and
  extended ed.}}
\newblock {de Gruyter Studies in Mathematics. Berlin: Walter de Gruyter},
  2011.

\bibitem{FT95}
M.~Fukushima and M.~Tomisaki.
\newblock {Reflecting diffusions on Lipschitz domains with cusps -- analytic
  construction and Skorohod representation.}
\newblock {\em Potential Anal.}, 4(4):377--408, 1995.

\bibitem{FT96}
M.~Fukushima and M.~Tomisaki.
\newblock {Construction and decomposition of reflecting diffusions on Lipschitz
  domains with H\"older cusps.}
\newblock {\em Probab. Theory Relat. Fields}, 106(4):521--557, 1996.

\bibitem{Hoh94}
R.~H\"ohnle.
\newblock {Construction of local solutions to SDE's with singular drift.}
\newblock {\em {Stochastics Stochastics Rep.}}, 47(3-4):163--192, 1994.

\bibitem{Hoh96}
R.~H\"ohnle.
\newblock {On global existence of solutions of SDE's with singular drift.}
\newblock {\em {Math. Nachr.}}, 179:145--160, 1996.

\bibitem{KS91}
I.~Karatzas and S.~E. Shreve.
\newblock {\em {Brownian motion and stochastic calculus. 2nd ed.}}
\newblock {Graduate Texts in Mathematics. New York etc.: Springer-Verlag}, 1991.


\bibitem{KrRo05}
{N.~V. Krylov and M.~R\"ockner}
\newblock {Strong solutions of stochastic equations with singular time
  dependent drift.}
\newblock {\em Probab. Theory Relat. Fields}, 131(2):154--196, 2005.

\bibitem{LS84}
P.-L. Lions and A.S. Sznitman.
\newblock {Stochastic differential equations with reflecting boundary
  conditions.}
\newblock {\em Commun. Pure Appl. Math.}, 37:511--537, 1984.

\bibitem{MR92}
{ Z.~Ma and M.~R{\"o}ckner}
\newblock {\em {Introduction to the theory of (non-symmetric) Dirichlet
  forms.}}
\newblock {Universitext. Berlin: Springer-Verlag}, 1992.


\bibitem{PaWi94}
E.~Pardoux and R.J. Williams.
\newblock {Symmetric reflected diffusions.}
\newblock {\em Ann. Inst. Henri Poincar\'e, Probab. Stat.}, 30(1):13--62, 1994.

\bibitem{Sai87}
Y.~Saisho.
\newblock {Stochastic differential equations for multi-dimensional domain with
  reflecting boundary.}
\newblock {\em Probab. Theory Relat. Fields}, 74:455--477, 1987.

\bibitem{SV71}
D.~W. Stroock and S.R.S. Varadhan.
\newblock {Diffusion processes with boundary conditions.}
\newblock {\em Commun. Pure Appl. Math.}, 24:147--225, 1971.


\bibitem{ST14}
J.~Shin and G.~Trutnau.
\newblock {On the stochastic regularity of distorted Brownian motions.}
\newblock{\em Trans. Amer. Math. Soc.}, https://doi.org/10.1090/tran/6887, 2016.


\bibitem{Tan79}
H.~Tanaka.
\newblock {Stochastic differential equations with reflecting boundary condition
  in convex regions.}
\newblock {\em Hiroshima Math. J.}, 9:163--177, 1979.

\bibitem{TP86}
R.J.A. Tough, P.N. Pusey, H.N.W. Lekkerkerker, and C.~Van Den~Broeck.
\newblock Stochastic descriptions of the dynamics of interacting brownian
  particles.
\newblock {\em Molecular Physics}, 59(3):595--619, 1986.

\bibitem{Tru03}
G. Trutnau.
\newblock {Skorokhod decomposition of reflected diffusions on bounded Lipschitz
  domains with singular non-reflection part.}
\newblock {\em Probab. Theory Relat. Fields}, 127(4):455--495, 2003.

\bibitem{WiZh90}
R.~J. Williams and W.~A. Zheng.
\newblock {On reflecting Brownian motion - a weak convergence approach.}
\newblock {\em Ann. Inst. Henri Poincar\'e, Probab. Stat.}, 26(3):461--488,
  1990.

\bibitem{Zw69}
R.~Zwanzig.
\newblock Langevin theory of polymer dynamics in dilute solution.
\newblock {\em Stochastic Processes in Chemical Physics}, 50:325--331, 1969.

\end{thebibliography}
\end{document}